\documentclass[12pt,english]{amsart}

\usepackage{fullpage}
\usepackage{stmaryrd}
\usepackage{appendix}
\usepackage{hyperref}
\usepackage{comment}

\newcommand{\Qq}{\mathbb{Q}} 
\newcommand{\Zz}{\mathbb{Z}} 
\newcommand{\Rr}{\mathbb{R}} 
\newcommand{\cc}{{\mathbb{C}}}    
\newcommand{\qq}{{\mathbb{Q}}}  
\newcommand{\zz}{{\mathbb{Z}}}  

\numberwithin{equation}{section}

{\theoremstyle{plain}
\newtheorem{theorem}{Theorem}[section]
\newtheorem{lemma}[theorem]{Lemma}
\newtheorem{proposition}[theorem]{Proposition}
\newtheorem{question}[theorem]{Question}
\newtheorem{corollary}[theorem]{Corollary}}

{\theoremstyle{remark}
\newtheorem{remark}[theorem]{Remark}
\newtheorem{definition}[theorem]{Definition}
\newtheorem{example}[theorem]{Example}}

\title{Density results for specialization sets of Galois covers}

\author{Joachim K\"onig}
\email{jkoenig@kaist.ac.kr}

\author{Fran\c cois Legrand}
\email{francois.legrand@tu-dresden.de}

\address{Department of Mathematical Sciences, KAIST, 291 Daehak-ro Yuseong-gu Daejeon 34141, South Korea}

\address{Institut f\"ur Algebra, Fachrichtung Mathematik, TU Dresden, 01062 Dresden, Germany}

\date{\today}

\begin{document}

\maketitle

\begin{abstract}
We provide evidence for this conclusion: given a finite Galois cover $f: X \rightarrow \mathbb{P}^1_\Qq$ of group $G$, almost all (in a density sense) realizations of $G$ over $\Qq$ do not occur as spe\-ci\-a\-lizations of $f$. We show that this holds if the number of branch points of $f$ is sufficiently large, under the abc-conjecture and, possibly, the lower bound predicted by the Malle conjecture for the number of Galois extensions of $\Qq$ of given group and bounded discriminant. This widely extends a result of Granville on the lack of $\Qq$-rational points on quadratic twists of hyperelliptic curves over $\Qq$ with large genus, under the abc-conjecture (a diophantine reformulation of the case $G=\Zz/2\Zz$ of our result). As a further evidence, we exhibit a few finite groups $G$ for which the above conclusion holds unconditionally for almost all covers of $\mathbb{P}^1_\Qq$ of group $G$. We also introduce a local-global principle for specializations of Galois covers $f: X \rightarrow \mathbb{P}^1_\Qq$ and show that it often fails if $f$ has abelian Galois group and sufficiently many branch points, under the abc-conjecture. On the one hand, such a local-global conclusion underscores the ``smallness" of the specialization set of a Galois cover of $\mathbb{P}^1_\Qq$. On the other hand, it allows to generate conditionally ``many" curves over $\Qq$ failing the Hasse principle, thus generalizing a recent result of Clark and Watson devoted to the hyperelliptic case.
\end{abstract}

\section{Introduction} \label{sec:intro}

Given a finite Galois extension $E$ of the rational function field $\Qq(T)$, and a point $t_0 \in \mathbb{P}^1(\Qq)$, there is a well-known notion of {\it{specialization}} $E_{t_0}/\Qq$ (see \S\ref{sssec:basics_1.1} for more details). If $E$ is the splitting field of a monic separable polynomial $P(T,Y) \in \Qq[T][Y]$ and $t_0 \in \Qq$ is such that $P(t_0,Y)$ is separable, then the field $E_{t_0}$ is the splitting field over $\Qq$ of $P(t_0,Y)$.

The specialization process has been much studied towards the {\it{inverse Galois problem}}, which asks whether every finite group $G$ occurs as the Galois group of a finite Galois extension $F/\Qq$. In that case, we shall say that such an extension $F/\Qq$ is a {\it{$G$-extension}}. Indeed, if $E/\Qq(T)$ is a finite Galois extension with Galois group $G$, then {\it{Hilbert's irreducibility theorem}} asserts that the specialization $E_{t_0}/\Qq$ still has Galois group $G$ for infinitely many $t_0 \in \Qq$. Moreover, if $E/\Qq(T)$ is {\it{$\Qq$-regular}} (i.e., if $\Qq$ is algebraically closed in $E$), in which case we shall say that $E/\Qq(T)$ is a {\it{regular $G$-extension}}, and if $G \not= \{1\}$, then infinitely many linearly disjoint $G$-extensions of $\Qq$ occur as specializations of $E/\Qq(T)$. In fact, most known realizations over $\Qq$ of finite non-abelian simple groups $G$ have been obtained by specializing regular $G$-extensions of $\Qq(T)$, generally derived from the {\it{rigidity method}}. See the books \cite{Ser92, Vol96,FJ08, MM18} for more details and references within.

Recent progress has been made on the set ${\rm{Sp}}(E)$ of all specializations of a given regular $G$-extension $E/\Qq(T)$. For example, for many groups $G$, no regular $G$-extension $E/\Qq(T)$ is {\it{pa\-rametric}}, i.e., ${\rm{Sp}}(E)$ does not contain all $G$-extensions of $\Qq$ (see \cite{KL18} and \cite[\S7]{KLN19}). Another result by D\`ebes \cite{Deb17} gives a lower bound for the number of $G$-extensions of $\Qq$ with bounded discriminant lying in the set ${\rm{Sp}}(E)$ for a given regular $G$-extension $E/\Qq(T)$. An even more fundamental question was raised in \cite{Deb18,DKLN18}: does the set ${\rm{Sp}}(E)$, a collection of arithmetic objects, characterize the extension $E/\Qq(T)$, a geometric one?

\subsection{A central question} \label{ssec:intro_0}

Given a regular $G$-extension $E/\Qq(T)$, the main purpose of this paper is to further study the set ${\rm{Sp}}(E)$ and to provide evidence for this striking conclusion: this set is in general ``small", i.e., ``almost all" $G$-extensions of $\Qq$ do not lie in the set ${\rm{Sp}}(E)$.

Let us make this more precise. Given an integer $x \geq 1$, let $\mathcal{S}(G,x)$ denote the set of all $G$-extensions $F/\Qq$ such that $|d_F| \leq x$, where $d_F$ denotes the absolute discriminant of $F$. By Hermite's theorem, the set $\mathcal{S}(G,x)$ is finite. Moreover, say that the set ${\rm{Sp}}(E)$ of all specializations of a given regular $G$-extension $E/\Qq(T)$ is {\it{of density zero}} if the equality $|{\rm{Sp}}(E) \cap {\mathcal{S}}(G,x)| = o(|{\mathcal{S}}(G,x)|)$ holds as $x$ tends to $\infty$.

\begin{question} \label{ques:density0}
Let $G$ be a finite group. Is it true that the specialization set ${\rm{Sp}}(E)$ of a given regular $G$-extension $E/\Qq(T)$, not in some ``small" exceptional list, is of density zero?
\end{question}

The reason why we have to consider an exceptional list in Question \ref{ques:density0} is that, for some regular $G$-extensions $E/\Qq(T)$, the specialization set ${\rm{Sp}}(E)$ is not of density zero. For example, this happens for all parametric extensions $E/\Qq(T)$, in which case a fully opposite conclusion holds. However, all extensions which are known to satisfy this property are in fact {\it{generic}} (that is, remain parametric after every base change) and, in particular, are all of genus 0 and belong to a very short list (see \cite[Theorem 1.6]{DKLN18} for more details).

In addition to the generic extensions $E/\Qq(T)$, there are some more counterexamples in genus 1. For instance, results of Vatsal \cite{Vat98}, Byeon \cite{Bye04}, and later Byeon-Jeon-Kim \cite{BJK09} about rank $1$ quadratic twists of elliptic curves yield infinite families of separable degree 3 polynomials $P(T) \in \Zz[T]$ such that a positive proportion of all quadratic extensions of $\Qq$ occur as specializations of the extension $\Qq(T)(\sqrt{P(T)})/\Qq(T)$. More generally, under Goldfeld's conjecture, 50$\%$ of all quadratic extensions of $\Qq$ are expected to be reached by specializing the function field extension corresponding to an elliptic curve over $\Qq$.

However, we are not aware of any counterexample in genus at least 2, and we in fact expect the answer to Question \ref{ques:density0} to be ``Yes" if regular $G$-extensions of $\Qq(T)$ of genus at most 1, which are quite rare and do not even exist for many finite groups $G$ (e.g., for all finite non-solvable groups), are left aside. Evidence for this is provided by \cite{DKLN18}, which proves an analog over the rational function field $\cc(Y)$,
with the notion of specialization replaced by a geometric analog of ``rational pullback" and the notion of density also replaced by a geometric analog via the Zariski topology.

In this paper, we make progress on Question \ref{ques:density0} in several directions. Firstly, in  \S\ref{sec:densityI}, we show that the answer is affirmative if one excludes regular $G$-extensions of $\Qq(T)$ with very few branch points, conditionally on widely accepted conjectures (see \S\ref{ssec:intro1}). Secondly, in \S\ref{sec:densityII}, we show for some exemplary small finite groups $G$ that, upon ignoring a ``small" (in a density sense) set of regular $G$-extensions of $\Qq(T)$, the answer to Question \ref{ques:density0} is positive unconditionally (see \S\ref{ssec:intro_3}). In this latter context, we do not have any restriction on the number of branch points or the genus, thus suggesting that the density zero conclusion, which we expect to hold always in genus at least 2, may also hold for ``many" regular $G$-extensions of $\Qq(T)$ of genus at most 1. For example, it is plausible that this conclusion holds for all genus 0 extensions which are not parametric.

\subsection{Conditional results} \label{ssec:intro1}

We first give an upper bound for the number of specializations of a given regular $G$-extension of $\Qq(T)$ with bounded discriminant, under the abc-conjecture:

\vspace{2mm}

\noindent
{\bf{The abc-conjecture.}} {\it{For every $\epsilon>0$, there exists a positive constant $K(\epsilon)$ such that, for all coprime integers $a$, $b$, and $c$ fulfilling $a+b=c$, the following holds:
$$c \leq K(\epsilon) \cdot {\rm{rad}}(abc)^{1+\epsilon},$$
where the radical ${\rm{rad}}(n)$ of an integer $n \geq 1$ is the product of the distinct prime factors of $n$.}}

\begin{theorem} \label{thm:intro_1}
Let $G$ be a finite group and $E/\Qq(T)$ a regular $G$-extension with $r \geq 5$ branch points. Suppose the abc-conjecture holds. Then there is a ``small" constant $e > 0$, depending only on $r$, $|G|$, and the ramification indices of the branch points of $E/\Qq(T)$, such that the following holds. For every $\epsilon >0$ and every sufficiently large integer $x$, one has
\begin{equation} \label{eq:intro_1}
|{\rm{Sp}}(E) \cap \mathcal{S}(G,x)| \leq x^{e+ \epsilon}.
\end{equation}
\end{theorem}

\noindent
See Theorem \ref{thm:abc_spec} for a more precise statement where we relax the lower bound on $r$ and give the precise definition of the exponent $e$.

To show that the specialization set of a given regular $G$-extension of $\Qq(T)$ with sufficiently many branch points is of density 0 (under the abc-conjecture), it then suffices, by \eqref{eq:intro_1}, to show that $|\mathcal{S}(G,x)|$ is asymptotically ``bigger" than $x^e$. A main difficulty to get this conclusion is that the asymptotic behaviour of $|\mathcal{S}(G,x)|$ is widely unknown for arbitrary finite groups $G$. However, general conjectures are available in the literature. 

For example, the Malle conjecture \cite{Mal02}, a classical landmark in this context, asserts that if $k$ is a number field and $G$ a finite group, then the number of $G$-extensions of $k$ whose relative discriminant has norm at most $x$ is roughly asymptotic to $x^{\alpha(G)}$, for some well-defined constant $\alpha(G)$ (recalled in \eqref{eq:intro_3} below). See \cite{Mal02} for more details and \cite[\S1.1]{Deb17} for a recent review of the state-of-the-art on the conjecture and its generalizations.

We only recall in details the lower bound predicted by the conjecture (in the specific case $k=\Qq$), which is enough for our purposes:

\vspace{2mm}

\noindent
{\bf{The Malle conjecture (lower bound).}} {\it{Let $G$ be a non-trivial finite group and let $p$ be the smallest prime divisor of $|G|$. Then there exists a positive constant $c(G)$ such that
\begin{equation} \label{conj:malle_lower}
c(G) \cdot x^{\alpha(G)} \leq |\mathcal{S}(G,x)|
\end{equation}
for every sufficiently large integer $x$, where
\begin{equation} \label{eq:intro_3}
\alpha(G)= \frac{1}{|G|} \cdot \frac{p}{p-1}.
\end{equation}}}

\vspace{-1mm}

\noindent
Note that if the lower bound \eqref{conj:malle_lower} holds for a given  finite group $G$ (for sufficiently large $x$), then $G$ occurs as a Galois group over $\Qq$.

The combination of \eqref{eq:intro_1} and \eqref{conj:malle_lower} then allows us to give this answer to Question \ref{ques:density0}:

\begin{theorem} \label{thm:intro_2}
Let $G$ be a finite group and $E/\Qq(T)$ a regular $G$-extension with $r \geq 7$ branch points. Suppose \eqref{conj:malle_lower} is fulfilled for the group $G$ and the abc-conjecture holds. Then the set of specializations of $E/\Qq(T)$ is of density zero.
\end{theorem}

\noindent
See Corollary \ref{coro:abc+malle} for a more precise statement. It should be pointed out that the bound $r \geq 7$ is not sharp (towards a density zero conclusion for every regular $G$-extension of $\Qq(T)$ of genus at least 2). For example, we can easily drop to $r \geq 6$ if $|G|$ is odd or to $r \geq 5$ if $|G|$ is prime to 6. Moreover, we obtain a conditional linear bound (depending on $G$) on the genus of a given regular $G$-extension $E/\Qq(T)$ for the set ${\rm{Sp}}(E)$ being of density 0. See Remark \ref{rk:eq3} for more details.

The bound \eqref{conj:malle_lower} is known to hold for several finite groups, thus providing concrete situations for which Theorem \ref{thm:intro_2} can be worded without mentioning it. For instance, relying on Shafarevich's theorem solving the inverse Galois problem for solvable groups, Kl\"uners and Malle \cite{KM04} proved the (lower bound of the) Malle conjecture for nilpotent groups. Another example is given by dihedral groups of order $2p$ with $p$ an odd prime, as proved by Kl\"uners in \cite{Klu06}. Moreover, many finite groups $G$ are such that every regular $G$-extension of $\Qq(T)$ has at least 7 branch points, thus yielding examples of groups $G$ for which the specialization set of {\it{every}} regular $G$-extension of $\Qq(T)$ is of density zero, under the abc-conjecture and, possibly, the lower bound \eqref{conj:malle_lower}. Such considerations are collected in Corollary \ref{coro:abc_explicit}. 

Although there is no known counterexample, the bound \eqref{conj:malle_lower} remains widely open, e.g., for most non-solvable groups. In the sequel, we give a variant of Theorem \ref{thm:intro_2} which applies to all finite groups, where the assumption that \eqref{conj:malle_lower} holds is not needed but where the bound on the number of branch points is less explicit. See Theorem \ref{thm:compare} for more details. This uses the already mentioned result of D\`ebes \cite{Deb17}, whose aim was to provide an unconditional weak version of the bound \eqref{conj:malle_lower} for regular Galois groups over $\Qq$ (i.e., for finite groups $G$ such that there is a regular $G$-extension of $\Qq(T)$), obtained by considering $G$-extensions of $\Qq$ which arise as specializations of a single regular $G$-extension of $\Qq(T)$. It should be pointed out that, by Theorem \ref{thm:intro_2}, one cannot hope (for arbitrary finite groups $G$) to obtain the exact bound \eqref{conj:malle_lower} in this way, thereby showing the limitations of the approach in \cite{Deb17}.

As a further result, we give a second variant, where the abc-conjecture is not required and no assumption on the number of branch points is made, provided the uniformity conjecture\footnote{which asserts that the number of $\Qq$-rational points on any given smooth curve over $\Qq$ of genus at least 2 is bounded by a quantity which depends only on the genus of the curve (but not on the curve itself).} holds and the upper bound from the Malle conjecture for some quotient of the underlying Galois group is taken into account (see Theorem \ref{thm:uniform}). As under the abc-conjecture, we may derive explicit examples of finite groups $G$ for which the specialization set of every regular $G$-extension of $\Qq(T)$ is of density zero, under the uniformity conjecture (see Corollary \ref{coro:uniform}). Note that, as Theorem \ref{thm:intro_2} and its consequences, Corollary \ref{coro:uniform} easily provides density zero conclusions for regular $G$-extensions of $\Qq(T)$ with few branch points.

\subsection{Unconditional results} \label{ssec:intro_3}

We start with the quadratic case. In the work \cite{Leg18b}, it was proved that, for ``almost all" regular $\Zz/2\Zz$-extensions $E/\Qq(T)$, at least one quadratic extension of $\Qq$ is not in ${\rm{Sp}}(E)$. Here, we sharpen this result as follows:

\begin{theorem} \label{thm:intro_3} 
Given an even positive integer $r$, the proportion of all regular $\Zz/2\Zz$-extensions $E/\Qq(T)$ with $r$ branch points, ``height" at most $H$, and whose set of specializations is of density $0$ tends to $1$ as $H$ tends to $\infty$. 
\end{theorem}

\noindent
From a diophantine point of view, this means that ``most" quadratic twists of ``most" hyperelliptic curves over $\Qq$ have only trivial $\Qq$-rational points, unconditionally (see Proposition \ref{prop:Z/2Z}(b)). See Theorems \ref{thm:Z/2Z_even} and \ref{thm:super_uncond} for more precise statements, and \S\ref{ssec:intro_5} for diophantine considerations in a more general context.

On the one hand, Theorem \ref{thm:intro_3} shows that invoking the abc-conjecture in the case $G=\zz/2\zz$ of Theorem \ref{thm:intro_2} is only necessary for comparatively few extensions. On the other hand, it shows that even among regular $\Zz/2\Zz$-extensions of $\Qq(T)$ to which Theorem \ref{thm:intro_2} does not apply (i.e., those with $r \leq 6$ branch points), only a few can be exceptions in Question \ref{ques:density0}\footnote{Clearly, there are exceptions in the case $r=2$ and, as already recalled in \S\ref{ssec:intro_0}, exceptions also exist in the case $r=4$. However, no exception seems to be expected in the case $r=6$ \cite[Conjecture 1]{Gra07}.}.

The second example we discuss is the symmetric group $S_3$. In this context, we have this result (see Theorem \ref{thm:S3} for a more precise statement):

\begin{theorem} \label{thm:intro_5}
Let $D$ be a positive integer. Then inside the set of all polynomials $P (T , Y ) = Y^3 + a(T )Y^2 + b(T )Y + c(T )$ with $a(T), b(T), c(T) \in \zz[T]$ of degree $\le D$ and of height $\le H$, the set of those $P(T,Y)$ which additionally define a regular $S_3$-extension of $\Qq(T)$ whose specialization set is of density zero, makes up a proportion tending to $1$ as $H$ tends to $\infty$.
\end{theorem}

\subsection{Comparison with previous non-parametricity results} \label{ssec:intro_3.5}

As already said, it was known from \cite{KL18} and \cite[\S7]{KLN19} that many finite groups $G$ do not have any parametric extension $E/\Qq(T)$. However, our results sharpen conditionally this conclusion. Indeed, by Theorem \ref{thm:intro_2}, for many finite groups $G$, not only at least one $G$-extension of $\Qq$ but actually almost all of them are not specializations of a given regular $G$-extension of $\Qq(T)$, under the abc-conjecture. Moreover, this yields (conditional) new examples of finite groups with no parametric extension $E/\Qq(T)$ (see Remark \ref{rk:abc_explicit}(b)). Furthermore, a property shared by the groups $\Zz/2\Zz$ and $S_3$ is that they admit a parametric extension $E/\Qq(T)$. Theo\-rems \ref{thm:intro_3} and \ref{thm:intro_5} show that if $G=\Zz/2\Zz$ or $S_3$, then parametric realizations are rare and, for almost all regular $G$-extensions $E/\Qq(T)$, the same fully opposite conclusion on the size of ${\rm{Sp}}(E)$ holds.

\subsection{Local-global considerations} \label{ssec:intro_4}

Our conditional results are global results, in the sense that they depend on diophantine properties and the arithmetic of curves over $\Qq$. On the contrary, our unconditional results are mostly due to local arguments. Namely, given a re\-gular $G$-extension $E/\Qq(T)$, let ${\rm{Sp}}(E)^{\rm{loc}}$ be the set of all $G$-extensions $F/\Qq$ such that $F\Qq_p/\Qq_p$ is a specialization of $E\Qq_p/\Qq_p(T)$ for all primes $p$ (including $p=\infty$, in which case $\Qq_p = \Rr$). Our local arguments consist in proving that, for almost all regular $G$-extensions $E/\Qq(T)$,
\begin{equation} \label{eq:ratio}
\frac{|{\rm{Sp}}(E)^{\rm{loc}} \cap {\mathcal{S}}(G,x)|}{|{\mathcal{S}}(G,x)|}
\end{equation}
tends to 0 as $x$ tends to $\infty$, thus yielding, in particular, that ${\rm{Sp}}(E)$ is of density 0. That is, almost all $G$-extensions $F/\Qq$ are not specializations of $E/\Qq(T)$ as this is wrong even up to base change from $\Qq$ to $\Qq_p$ (for at least one suitable prime $p$ depending on $F$). This suggests this refinement of Question \ref{ques:density0}, which asks whether the specialization set of $E/\Qq(T)$ is of density 0 even within the set of those $G$-extensions of $\Qq$ who pass these local obstructions:

\begin{question} \label{ques:Hasse2}
Let $G$ be a finite group. Does it hold that, for a given regular $G$-extension $E/\Qq(T)$, not in some exceptional list, the ratio
\begin{equation} \label{eq:ratio_2}
\frac{|{\rm{Sp}}(E) \cap {\mathcal{S}}(G,x)|}{|{\rm{Sp}}(E)^{\rm{loc}} \cap {\mathcal{S}}(G,x)|}
\end{equation}
tends to 0 as $x$ tends to $\infty$?
\end{question}

\noindent
A positive answer means that there exist ``many" $G$-extensions of $\Qq$ which are not specializations of $E/\Qq(T)$, but this cannot be detected by local considerations, implying the failure of a local-global principle for specializations.

In \S\ref{sec:hasse}, we prove the following result, which provides some evidence for a positive answer to Question \ref{ques:Hasse2} and strengthens the conclusion of Theorem \ref{thm:intro_2} for abelian groups:

\begin{theorem} \label{thm:intro_6}
Let $G$ be a finite abelian group and $E/\Qq(T)$ a regular $G$-extension with $r \geq 7$ branch points. Then the ratio \eqref{eq:ratio_2} tends to 0 as $x$ tends to $\infty$, under the abc-conjecture.
\end{theorem}

\noindent
See Theorem \ref{thm:hasse_ab} for a more general result which applies to any finite group $G$ with non-trivial center and to any regular $G$-extension $E/\Qq(T)$ with $r \geq 8$ branch points ($r \geq 7$ is sufficient for abelian groups) and suitable geometric inertia groups.

\subsection{Diophantine aspects} \label{ssec:intro_5}

In \S\ref{sec:diophantine}, we discuss diophantine aspects of our results, whose most general versions in the sequel are actually worded in terms of Galois covers of $\mathbb{P}^1$.

Given a regular Galois cover $f : X \rightarrow \mathbb{P}^1_\Qq$ with Galois group $G$ and a (continuous) epimorphism $\varphi : {\rm{G}}_\Qq  \rightarrow G$, where G$_\Qq$ denotes the absolute Galois group of $\Qq$, there is a notion of {\it{twisted cover}} $\widetilde{f}^\varphi : \widetilde{X}^\varphi \rightarrow \mathbb{P}^1_\Qq$, introduced by D\`ebes in \cite{Deb99a}, which satisfies this property: $\varphi$ is a specialization morphism of $f$ \footnote{A refined version of ``a $G$-extension $F/\Qq$ occurs as a specialization of a regular $G$-extension $E/\Qq(T)$".} if and only if $\widetilde{X}^\varphi$ has a {\it{non-trivial}} $\Qq$-rational point, i.e., a $\Qq$-rational point which does not extend any branch point of $f$. See \S\ref{ssec:diophantine_prelim} for more details.

Hence, the most general versions of Theorems \ref{thm:intro_1}-\ref{thm:intro_5} can be stated with this diophantine flavour. For example, the corresponding variant of Theorem \ref{thm:intro_1} provides, for a regular Galois cover $f : X \rightarrow \mathbb{P}^1_\Qq$ of group $G$ with $r \geq 5$ branch points, an upper bound for the number of epimorphisms $\varphi : {\rm{G}}_\Qq \rightarrow G$ of bounded discriminant such that the twisted curve $\widetilde{X}^\varphi$ has at least one non-trivial $\Qq$-rational point. See Theorem \ref{thm:abc_spec2} for a more precise statement. The special case $G=\Zz/2\Zz$ of our result is nothing but a well-known result of Granville \cite[Corollary 1]{Gra07} on the number of quadratic twists of a given hyperelliptic curve over $\Qq$ of genus at least 2 with non-trivial $\Qq$-rational points, under the abc-conjecture (see Corollary \ref{coro:granville}).

Similarly, the same applies to Theorem \ref{thm:intro_6}. Given a regular Galois cover $f: X \rightarrow \mathbb{P}^1_\Qq$ of group $G$, the existence of an epimorphism $\varphi : {\rm{G}}_\Qq \rightarrow G$ which occurs as a specialization morphism of $f$ everywhere locally but not globally means that the twisted curve $\widetilde{X}^\varphi$ has a non-trivial $\Qq_p$-rational point for every prime $p$ but only trivial $\Qq$-rational points. This diophantine reformulation of the failure of our local-global principle for specializations is actually strictly identical to the failure of the Hasse principle for curves, provided $f$ has no $\Qq$-rational branch point. In particular, the diophantine analog of Theorem \ref{thm:intro_6} provides the following:

\begin{theorem} \label{thm:intro_7}
{\it{Let $C$ be a $\Qq$-curve with a finite abelian cover $f$ to $\mathbb{P}^1$ such that $f$ has at least 7 branch points and $f$ has no $\Qq$-rational branch point. Assume the abc-conjecture holds. Then there exist ``many" $\Qq$-curves $C'$, which are isomorphic to $C$ up to base change from $\Qq$ to $\overline{\Qq}$ and which do not fulfill the Hasse principle.}}
\end{theorem}

\noindent
See Corollary \ref{thm:hasse_exp} for a more general result, which also applies in some non-abelian situations, and Corollary \ref{coro:hasse_existence} for a variant which in fact applies to any regular Galois group over $\Qq$ with non-trivial center, at the cost of choosing the curve $C$ more suitably. In the quadratic case, our results allow to retrieve a recent result of Clark and Watson \cite[Theorem 2]{CW18}, which asserts that ``many" quadratic twists of a hyperelliptic curve $C:y^2=P(t)$ with $P(T) \in \Zz[T]$ separable, of even degree $\geq 8$, and with no root in $\Qq$ do not fulfill the Hasse principle, under the abc-conjecture (see Corollary \ref{coro:CW18}).

\vspace{2mm}

{\bf Acknowledgements.} The first author was supported by the National Research Foundation of Korea (NRF Grant no. 2019002665). We also wish to thank Arno Fehm for his help with Lemma \ref{lemma_0}.

\section{Basics} \label{sec:basics}

The aim of this section is fourfold. \S\ref{ssec:basics_0} is devoted to some general notation we shall use in the sequel. In \S\ref{ssec:basics_1}, we recall classical material about Galois covers of the projective line while \S\ref{ssec:basics_2} is devoted to rational points on superelliptic curves. As to \S\ref{ssec:basics_3}, we there make the content of \S\ref{ssec:basics_1} explicit in the quadratic case, in relation with the material from \S\ref{ssec:basics_2}.

\subsection{General notation} \label{ssec:basics_0}

Denote the absolute Galois group of a field $k$ of characteristic zero by G$_k$. If $k'$ is a field containing $k$, we use the notation $\otimes_k \, k^\prime$ for the scalar extension from $k$ to $k'$. For example, if $X$ is a $k$-curve, then $X\otimes_kk'$ is the $k'$-curve obtained by scalar extension from $k$ to $k'$. Conjugation automorphisms in a group $G$ are denoted by ${\rm{conj}}(\omega)$ for $\omega \in G$: ${\rm{conj}}(\omega)(x) = \omega x \omega^{-1}$ ($x \in G$).

Let $n \geq 2$, $N \geq 1$, $x \geq 1$, and $H \geq 1$  be integers. Let $G$ be a finite group and $T$ an indeterminate. We use the following notation:

\noindent
(a) $\mathcal{S}(G,x)$: set of all $G$-extensions $F/\Qq$ such that $|d_F| \leq x$, where $d_F$ denotes the absolute discriminant of the number field $F$,

\noindent
(b) $\overline{\mathcal{S}}(G,x)$: set of all (continuous) epimorphisms $\varphi:{\rm{G}}_{\qq}\to G$ such that $\overline{\Qq}^{{\rm{ker}}(\varphi)}/\Qq \in \mathcal{S}(G,x)$, modulo the equivalence which identifies $\varphi$ and $\varphi'$ if $\varphi' = {\rm{conj}}(\omega) \circ \varphi$ for some $\omega \in G$ (the set $\overline{\mathcal{S}}(G,x)$ refines the set $\mathcal{S}(G,x)$ but note that the cardinalities are equal up to an explicit multiplicative constant depending only on $G$),

\noindent
(c) $\mathcal{N}_n$: set of all $n$-free integers, that is, of all integers $d$ such that $d \not \in \{0,1\}$ and $p^n$ divides $d$ for no prime number $p$ (if $n=2$, we say squarefree instead of $2$-free); recall that $\mathcal{N}_n$ has density $1/\zeta(n)$, where $\zeta$ denotes the Riemann-zeta function,

\noindent
(d) $\mathcal{N}_n(x)$: subset of $\mathcal{N}_n$ defined by the extra condition that $|d| \leq x$,

\noindent
(e) $\mathcal{P}(n,N)$: set of all degree $N$ polynomials $P(T) \in \Zz[T]$  whose roots have multiplicity $\leq n-1$,

\noindent
(f) $\mathcal{P}(n,N, H)$: subset of $\mathcal{P}(n,N)$ defined by the extra condition that the height is at most $H$; recall that the height of $a_0 + a_1 T + \cdots + a_N T^N$ is ${\rm{max}}(|a_0|, \dots, |a_N|)$,

\noindent
(g) $\mathcal{P}_2(n,N)$: subset of $\mathcal{P}(n,N)$ which consists of all elements $P(T)$ with squarefree content,

\noindent
(h) $\mathcal{P}_2(n,N, H) = \mathcal{P}_2(n,N) \cap \mathcal{P}(n,N, H)$.

\begin{definition} \label{def:density_general}
Let $B$ be a set, $(B_n)_{n \geq 1}$ an increasing sequence of finite subsets of $B$ such that $B= \cup_{n \geq 1} B_n$, and $A$ a subset of $B$. If 
$$\frac{|A \cap B_n|}{|B_n|}$$
tends to some $d \in [0,1]$ as $n$ tends to $\infty$, we say that $d$ is the {\it{density}} of the set $A$ (in $B$).
\end{definition}

Although this notion depends on the sequence $(B_n)_{n \geq 1}$, we do not make this dependency explicit in the terminology as our choices in the sequel will always be natural.

\subsection{Finite Galois covers of the projective line} \label{ssec:basics_1}

Let $k$ be a field of characteristic zero, $T$ an indeterminate, $\Omega$ an algebraic closure of $k(T)$, and $\overline{k}$ the algebraic closure of $k$ in $\Omega$.

\subsubsection{Generalities} \label{sssec:basics_1.1}

For more on the material below, we mostly refer to \cite[\S2.1]{DG12}.

A {\it{$k$-cover}} of $\mathbb{P}^1$ is a finite and generically unramified morphism $f:X \rightarrow \mathbb{P}^1$ defined over $k$, with $X$ a normal and irreducible $k$-curve. We make no distinction between a $k$-cover $f : X \rightarrow \mathbb{P}^1$ and the associated function field extension $E/k(T)$ (with $E \subseteq \Omega$): $f$ is the normalization of $\mathbb{P}^1$ in $E$ and $E$ is the function field $k(X)$ of $X$. The $k$-cover $f:X\rightarrow \mathbb{P}^1$ is said to be {\it{regular}} if $E$ is a regular extension of $k$ (i.e., if $E \cap \overline k=k$) or, equivalently, if $X$ is geometrically irreducible. We also say that the $k$-cover $f:X \rightarrow \mathbb{P}^1$ is {\it{Galois}} if $E/k(T)$ is. If, in addition, $G$ denotes the Galois group of $E/k(T)$, we say that $f$ is a {\it{$k$-$G$-cover}}. 

Fix a regular $k$-cover $f : X \rightarrow \mathbb{P}^1$ and denote its function field extension by $E/k(T)$.

A point $t_0 \in \mathbb{P}^1(\overline{k})$ is a {\it{branch point}} of $f$ (or of $E/k(T)$) if the prime ideal of $\overline{k}[T-t_0]$ generated by $T-t_0$ ramifies in the integral closure of $\overline{k}[T-t_0]$ in the compositum $\widehat{E}\overline{k}$ of $\widehat{E}$ and $\overline{k}(T)$ inside $\Omega$ (set $T-t_0 = 1/T$ if $t_0=\infty$), where $\widehat{E}$ denotes the Galois closure of $E$ over $k(T)$ inside $\Omega$. There are only finitely many branch points, denoted by $t_1, \dots, t_r$.

Suppose $f$ is Galois and set $G={\rm{Gal}}(E/k(T))$. Say that $E/k(T)$ is a {\it{regular $G$-extension}}.

Denote the $k$-fundamental group of $\mathbb{P}^1 \setminus \{t_1, \dots,t_r\}$ by $\pi_1(\mathbb{P}^1 \setminus \{t_1, \dots,t_r\}, t)_k$, where $t \in \mathbb{P}^1(\overline{k}) \setminus \{t_1, \dots,t_r\}$ is a base point. To the regular $k$-$G$-cover $f : X \rightarrow \mathbb{P}^1$ corresponds an epimorphism $\phi : \pi_1(\mathbb{P}^1 \setminus \{t_1, \dots,t_r\}, t)_k \rightarrow G$ whose restriction  to the $\overline{k}$-fundamental group $\pi_1(\mathbb{P}^1 \setminus \{t_1, \dots,t_r\}, t)_{\overline{k}}$ remains surjective.

Every $t_0 \in \mathbb{P}^1(k) \setminus \{t_1, \dots, t_r\}$ yields a section $s_{t_0} : {\rm{G}}_k \rightarrow \pi_1(\mathbb{P}^1 \setminus \{t_1, \dots,t_r\}, t)_k$ to the exact sequence
$$1 \rightarrow \pi_1(\mathbb{P}^1 \setminus \{t_1, \dots,t_r\}, t)_{\overline{k}} \rightarrow \pi_1(\mathbb{P}^1 \setminus \{t_1, \dots,t_r\}, t)_k \rightarrow {\rm{G}}_k \rightarrow 1,$$
which is uniquely defined up to conjugation by an element of $\pi_1(\mathbb{P}^1 \setminus \{t_1, \dots,t_r\}, t)_{\overline{k}}$. The homomorphism $\phi \circ s_{t_0} : {\rm{G}}_k \rightarrow G$ is denoted by $f_{t_0}$ and called the {\it{specialization morphism}} of $f$ at $t_0$. The fixed field in $\overline{k}$ of ${\rm{ker}}(f_{t_0})$ is the residue field at some prime ideal $\mathfrak{p}$ lying over the prime ideal of ${k}[T-t_0]$ generated by $T-t_0$ in the extension $E/k(T)$ \footnote{This does not depend on the choice of $\mathfrak{p}$ as the extension $E/k(T)$ is Galois.}. We denote it by $E_{t_0}$ and call the extension $E_{t_0}/k$ the {\it{specialization}} of $E/k(T)$ at $t_0$. The Galois group of $E_{t_0}/k$ is the decomposition group of $E/k(T)$ at a prime $\mathfrak{p}$ as above.

Let us define the following two sets:
$${\rm{Sp}}(f)=\{f_{t_0} : {\rm{G}}_k \rightarrow G \, \, : \, \, t_0 \in \mathbb{P}^1(k) \setminus \{t_1, \dots,t_r\} \},$$
$${\rm{Sp}}(E) = \{E_{t_0} /k \, \, : \, \, t_0 \in \mathbb{P}^1(k) \setminus \{t_1, \dots,t_r\}\}.$$

As a special case of Definition \ref{def:density_general}, say that $d \in [0,1]$ is the {\it{density}} of the set ${\rm{Sp}}(f)$ if
$$\dfrac{|{\rm{Sp}}(f) \cap \overline{\mathcal{S}}(G,x)|}{|\overline{\mathcal{S}}(G,x)|}$$
tends to $d$ as $x$ tends to $\infty$. We define analogously the density of the set ${\rm{Sp}}(E)$. Note that the set ${\rm{Sp}}(f)$ is of density $0$ if and only if the set ${\rm{Sp}}(E)$ is. 

Recall that $E/k(T)$ is {\it{parametric}} if every $G$-extension of $k$ lies in the set ${\rm{Sp}}(E)$, and that $E/k(T)$ is {\it{generic}} if $Ek'/k'(T)$ is parametric for every overfield $k' \supseteq k$.

\subsubsection{Ramification in specializations}

We review a well-known result relating the ramification of $f$ to that of its specializations. Keep the notation from \S\ref{sssec:basics_1.1} and take $k=\Qq$.

The {\it{minimal polynomial}} of $t=[a : b] \in  \mathbb{P}^1(\overline{\qq})$ is the unique (up to sign) homogeneous polynomial $P(U,V)\in \zz[U,V]$ defined as follows. If $b=0$, set $P(U,V) = V$. Otherwise, let $P(U,V)$ be the homogenization of the irreducible polynomial in $\zz[U]$ with root $a/b$. Given a prime number $p$, say that $t$ is {\it{$p$-integral}} if $p$ divides neither the coefficient of the leading $U$-term nor of the leading $V$-term of $P(U,V)$. If $t_0$ is in $\mathbb{P}^1(\qq) \setminus \{t\}$, set $t_0=[a':b']$
with $a'$ and $b'$ coprime integers. Define $I_p(t_0, t)$ as the $p$-adic valuation of $P(a', b')$ (if $t$ is $p$-integral).

The theorem below is an immediate consequence of a fundamental result of Beckmann \cite[Proposition 4.2]{Bec91} (see also \cite[\S2.2]{Leg16}):

\begin{theorem} \label{beckmann}
For every prime number $p$, not in some finite set $\mathcal{S}_{\rm{exc}}$ depending only on $f$, and every $t_0 \in \mathbb{P}^1(\Qq) \setminus \{t_1,\dots, t_r\}$, the following two conclusions hold.

\vspace{0.5mm}

\noindent
{\rm{(a)}} If $p$ ramifies in the specialization $E_{t_0}/\Qq$, then $I_{p}(t_0,t_{i})>0$ for some (necessarily unique up to $\Qq$-conjugation) ${i} \in \{1, \dots, r\}$.

\vspace{0.5mm}

\noindent
{\rm{(b)}} If $I_{p}(t_0,t_{i})>0$, then the inertia group of $E_{t_0}/\Qq$ at $p$ is conjugate in $G$ to $\langle \tau_{i}^{I_{p}(t_0,t_{i})} \rangle$, with $\tau_{i}$ a generator of an inertia subgroup of $E\overline{\Qq} / \overline{\Qq}(T)$ at the prime ideal generated by $T-t_{i}$.
\end{theorem}

\subsection{Superelliptic curves} \label{ssec:basics_2}

Let $n$ and $N$ be integers with $n \geq 2$ and $N \geq 1$. Set $N=qn+r$, with $q \geq 0$ and $0 \leq r \leq n-1$. Let 
$P(T)=a_0 + a_1 T + \cdots + a_{N-1} T^{N-1} + a_N T^N$ be in $\mathcal{P}(n,N)$.

\subsubsection{The case where $n$ divides $N$}

First, assume $r=0$. Consider this equivalence relation on $\overline{\Qq}^3 \setminus \{(0,0,0)\}$:
$(y_1,t_1,z_1) \sim (y_2,t_2,z_2)$
iff $(y_2,t_2,z_2) = (\lambda^{N/n}y_1,\lambda t_1, \lambda z_1)$ for some $\lambda \in \overline{\Qq} \setminus \{0\}$. The quotient space $(\overline{\Qq}^3 \setminus \{(0,0,0)\} )/ \sim$ is a weighted projective space, denoted by $\mathbb{P}_{N/n,1,1}(\overline{\Qq}).$ Given $(y,t,z) \in \overline{\Qq}^3 \setminus \{(0,0,0)\}$, the corresponding point in $\mathbb{P}_{N/n,1,1}(\overline{\Qq})$ is denoted by $[y:t:z].$ Set
$$P(T,Z)= a_0 Z^N + a_1 Z^{N-1} T + \cdots + a_{N-1} Z T^{N-1} +a_N T^N.$$
The equation $Y^n= P(T,Z)$ in $\mathbb{P}_{N/n,1,1}(\overline{\Qq})$ is {\it{the superelliptic\footnote{Here and in \S\ref{sssec:n_odd} below, say {\it{hyperelliptic}} if $n=2$.} curve associated with $P(T)$}}; we denote it by $C_{P(T)}.$ The set of all $\Qq$-rational points on $C_{P(T)}$, i.e., the set of all elements $[y:t:z] \in \mathbb{P}_{N/n,1,1}(\overline{\Qq})$ such that $(y,t,z) \in \Qq^3 \setminus \{(0,0,0)\}$ and $y^n=P(t,z)$, is denoted by 
$C_{P(T)}(\Qq).$ A point $[y:t:z] \in C_{P(T)}(\Qq)$ is {\it{trivial}} if $y=0$, and {\it{non-trivial}} otherwise. Equivalently, $[y:t:z] \in C_{P(T)}(\Qq)$ is trivial if $z \not=0$ and $P(t/z)=0$.

\subsubsection{The case where $n$ does not divide $N$} \label{sssec:n_odd}

Now, we consider the case $r \geq 1$, which is in fact similar to the previous one. However, to avoid confusion, we state it in details.

Consider this equivalence relation on $\overline{\Qq}^3 \setminus \{(0,0,0)\}$:
$(y_1,t_1,z_1) \sim (y_2,t_2,z_2)$
if and only if $(y_2,t_2,z_2) = (\lambda^{(N+n-r)/n}y_1,\lambda t_1, \lambda z_1)$ for some $\lambda \in \overline{\Qq} \setminus \{0\}$. The quotient space $(\overline{\Qq}^3 \setminus \{(0,0,0)\} )/ \sim$ is a weighted projective space, denoted by $\mathbb{P}_{(N+n-r)/n,1,1}(\overline{\Qq}).$ Given $(y,t,z) \in \overline{\Qq}^3 \setminus \{(0,0,0)\}$, the corresponding point in $\mathbb{P}_{(N+n-r)/n,1,1}(\overline{\Qq})$ is denoted by $[y:t:z].$ Set
$$P(T,Z)= a_0 Z^{(q+1)n} + a_1 Z^{(q+1)n-1} T + \cdots + a_{N-1} Z^{n-r+1} T^{N-1} +a_N Z^{n-r} T^N.$$
The equation $Y^n= P(T,Z)$ in $\mathbb{P}_{(N+n-r)/n,1,1}(\overline{\Qq})$ is {\it{the superelliptic curve associated with $P(T)$}}; we denote it by $C_{P(T)}.$ The set of all $\Qq$-rational points on $C_{P(T)}$, i.e., the set of all points $[y:t:z] \in \mathbb{P}_{(N+n-r)/n,1,1}(\overline{\Qq})$ such that $(y,t,z) \in \Qq^3 \setminus \{(0,0,0)\}$ and $y^n= P(t,z)$, is denoted by 
$C_{P(T)}(\Qq).$ A point $[y:t:z] \in C_{P(T)}(\Qq)$ is {\it{trivial}} if $y=0$, and {\it{non-trivial}} otherwise. Equivalently, $[y:t:z] \in C_{P(T)}(\Qq)$ is trivial if either $z=0$ (this point, which is $[0:1:0]$, is the point at $\infty$) or $z \not=0$ and $t/z$ is a root of $P(T)$. 

\subsubsection{Extra notation}

We use the following notation:

\vspace{1mm}

\noindent
(a) $\mathcal{N}_n(P(T))$: subset of $\mathcal{N}_n$ defined by the extra condition that the ``twisted" superelliptic curve $C_{d \cdot P(T)}: y^n=d \cdot P(t)$ has a non-trivial $\Qq$-rational point,

\vspace{1mm}

\noindent
(b) $\mathcal{N}_n(P(T),x) = \mathcal{N}_n(P(T)) \cap \mathcal{N}_n(x)$ ($x \geq 1$).

\subsection{On the quadratic case} \label{ssec:basics_3}

The following elementary proposition, which gives an explicit description of the set of branch points and characterizes specializations of a given regular $\Zz/2\Zz$-extension of $\Qq(T)$, will be needed in the sequel. See, e.g., \cite[\S8]{KL18} for a proof.

\begin{proposition} \label{prop:Z/2Z}
Let $N \geq 1$ be an integer and $P(T) \in \mathcal{P}(2,N)$. Denote the roots of $P(T)$ by $t_1, \dots, t_N$ and the field $\Qq(T)(\sqrt{P(T)})$ by $E$. 

\vspace{0.5mm}

\noindent
{\rm{(a)}} The set of branch points of $E/\Qq(T)$ is either the set $\{t_1,\dots,t_N\}$ (if $N$ is even) or the set $\{t_1,\dots,t_N\} \cup \{ \infty \}$ (if $N$ is odd).

\vspace{0.5mm}

\noindent
{\rm{(b)}} Let $d$ be in $\mathcal{N}_2$. Then the $\Zz/2\Zz$-extension $\Qq(\sqrt{d})/\Qq$ is in ${\rm{Sp}}(E)$ if and only if the twisted hyperelliptic curve $C_{d \cdot P(T)}: y^2 = d \cdot P(t)$ has a non-trivial $\Qq$-rational point.
\end{proposition}

Given an indeterminate $T$, there is a natural bijection $f$ between the set of all regular $\Zz/2\Zz$-extensions of $\Qq(T)$ and the set of all separable polynomials $P(T) \in \Zz[T]$ with squarefree content. Then define the {\it{height}} of a given regular $\Zz/2\Zz$-extension $E/\Qq(T)$ as the height of the associated polynomial $P_E(T)$. Moreover, by Proposition \ref{prop:Z/2Z}(a), if $r$ is a positive even integer, then $E/\Qq(T)$ has $r$ branch points if and only if $P_E(T)$ has degree $r$ or $r-1$.

Given positive integers $r$ and $H$ with $r$ even, we use the following notation:

\vspace{1mm}

\noindent
(a) $\mathcal{E}(r)$: set of all regular $\Zz/2\Zz$-extensions of $\Qq(T)$ with $r$ branch points,

\vspace{1mm}

\noindent
(b) $\mathcal{E}(r,H)$: subset of $\mathcal{E}(r)$ defined by the extra condition that the height is at most $H$.

\begin{proposition} \label{prop:card}
Given an even positive integer $r$, there exists a constant $\alpha(r) > 0$ such that
\begin{equation} \label{eq:card1}
|\mathcal{E}(r,H)| \sim | \mathcal{P}_2(2,r,H) | \sim \alpha(r) \cdot H^{r+1}, \quad H \to \infty,
\end{equation}
\begin{equation} \label{eq:card2}
\frac{|\mathcal{E}(r,H)| - | \mathcal{P}_2(2,r,H) |}{|\mathcal{E}(r,H)|} = O \bigg(\frac{1}{H} \bigg), \quad H \to \infty.
\end{equation}
\end{proposition}

\begin{proof}
Given $H \geq 1$, Proposition \ref{prop:Z/2Z}(a) shows that  
\begin{equation} \label{1.1}
|\mathcal{E}(r,H)|= | \mathcal{P}_2(2,r,H) | + | \mathcal{P}_2(2,r-1,H) |
\end{equation}
By \cite[Lemma 5.8]{Leg18b}, one has
\begin{equation} \label{1.2}
| \mathcal{P}_2(2,r,H) | \sim \alpha(r) \cdot H^{r+1}, \quad H \to \infty,
\end{equation}
for some positive constant $\alpha(r)$ and, clearly, one has
\begin{equation} \label{1.3}
| \mathcal{P}_2(2,r-1,H) |= O(H^r), \quad H \to \infty.
\end{equation}
It then remains to combine \eqref{1.1}, \eqref{1.2}, and \eqref{1.3} to get \eqref{eq:card1} and \eqref{eq:card2}, as needed.
\end{proof}

\section{Conditional results} \label{sec:densityI}

This section is devoted to our conditional results which assert that the specialization set of a regular $\qq$-$G$-cover of $\mathbb{P}^1$ with sufficiently many branch points has density zero.

We need some notation, in addition to that from \S\ref{sec:basics}. Let $G$ be a non-trivial finite group and $f:X\to \mathbb{P}^1$ a regular $\qq$-$G$-cover. We denote the associated regular $G$-extension by $E/\Qq(T)$. Let $S = \{t_1,\dots, t_r\} \subseteq \mathbb{P}^1(\overline{\qq})$ be a non-empty subset of the set of all branch points of $f$, closed under the action of G$_\Qq$. Denote the ramification index of $t_i$ by $e_i$, $i=1, \dots,r$, and set $e_0 = \min \{e_1,\dots,e_r\} $. Let $q_0$ be the smallest prime dividing one of the $e_i$'s and $p$ the smallest prime divisor of $|G|$.

\subsection{A conditional upper bound} \label{ssec:upbound}

This more precise version of Theorem \ref{thm:intro_1} gives an upper bound for $|{\rm{Sp}}(f) \cap \overline{\mathcal{S}}(G,x)|$, provided $r$ is large enough and the abc-conjecture holds:

\begin{theorem} \label{thm:abc_spec}
Assume the abc-conjecture holds and
\begin{equation} \label{eq1}
r > 2+\frac{2}{q_0-1}.
\end{equation}
Then, for every $\epsilon>0$ and every sufficiently large integer $x$, one has 
$$|{\rm{Sp}}(f) \cap  \overline{\mathcal{S}}(G,x)| \leq x^{e + \epsilon},$$
where
\begin{equation} \label{eq1.5}
e={2} \cdot |G|^{-1} \cdot \bigg(1-\frac{1}{e_0} \bigg)^{-1} \cdot \bigg(r-2-\frac{2}{q_0-1} \bigg)^{-1}.
\end{equation}
\end{theorem}

\begin{remark} \label{rk:stup}
(a) The set $S$, which is implicit in Theorem \ref{thm:abc_spec} as well as in the next result can most conveniently be chosen to be the set of all branch points of $f$. However, in some situations, proper subsets yield stronger conclusions, notably if there are many branch points with large ramification index. From the proof of Theorem \ref{thm:abc_spec} (see \S\ref{ssec:proof}), considering several subsets at the same time (with the corresponding notation for each subset) can sometimes yield even stronger results. We refrain from explicitly stating such a version of Theorem \ref{thm:abc_spec}, to avoid unnecessarily complicated notation.

\vspace{0.5mm}

\noindent	
(b) Condition \eqref{eq1} holds if and only if one of these conditions is satisfied:

\vspace{0.5mm}

(1) $r \geq 5$,

\vspace{0.5mm}

(2) $r\ge 4$ and $q_0\ge 3$,

\vspace{0.5mm}

(3) $r\ge 3$ and $q_0 \ge 5$.
\end{remark}

\subsection{Explicit examples} \label{ssec:coros}

We now explain how deriving several explicit results with the conclusion that the set ${\rm{Sp}}(f)$ has density zero.

First, we combine the lower bound given by the Malle conjecture and the upper bound from Theorem \ref{thm:abc_spec} to obtain the following more precise version of Theorem \ref{thm:intro_2}:

\begin{corollary} \label{coro:abc+malle}
Assume the lower bound \eqref{conj:malle_lower} is fulfilled for the group $G$, the abc-conjecture holds, and the following condition is satisfied:
\begin{equation} \label{eq2}
r > 2 \bigg(\frac{q_0}{q_0-1} + \frac{(p-1)e_0}{p(e_0-1)} \bigg),
\end{equation}
Then one has $e < \alpha(G)$, where $e$ and $\alpha(G)$ are defined in \eqref{eq1.5} and \eqref{eq:intro_3}, respectively, and, for every $\epsilon >0$ and every sufficiently large $x$, one has
\begin{equation} \label{eq:rate}
\frac{|{\rm{Sp}}(f) \cap  \overline{\mathcal{S}}(G,x)|}{| \overline{\mathcal{S}}(G,x)|} = O(x^{e  - \alpha(G) + \epsilon}).
\end{equation}
In particular, the set ${\rm{Sp}}(f)$ has density 0.
\end{corollary}

\begin{proof}
First, note that \eqref{eq2} $\Rightarrow$ \eqref{eq1} as
$$2 \bigg(\frac{q_0}{q_0-1} + \frac{(p-1)e_0}{p(e_0-1)} \bigg) > \frac{2q_0}{q_0-1} = 2+ \frac{2}{q_0-1}.$$
Then, by Theorem \ref{thm:abc_spec} and since \eqref{conj:malle_lower} has been assumed to hold, \eqref{eq:rate} holds. To complete the proof, it suffices to check $e < \alpha(G)$. Clearly, this holds if and only if \eqref{eq2} is satisfied.
\end{proof}

\begin{remark} \label{rk:eq3}
Making use of the inequalities $2\le p\le q_0\le e_0$, one sees that \eqref{eq2} holds as soon as one of the following conditions is satisfied:

\vspace{0.5mm}

\noindent
(a) $r \geq 7$,

\vspace{0.5mm}

\noindent
(b) $r= 6$ and $e_0\ge 3$,

\vspace{0.5mm}

\noindent
(c) $r= 5$, $q_0 \ge 3$, and $(e_0,q_0,p)\ne (3,3,3)$,

\vspace{0.5mm}

\noindent
(d) $r= 4$ and $q_0>2p$.

\vspace{0.5mm}

\noindent
Conversely, since the right-hand side of \eqref{eq2} is bounded from below by $3$, Corollary \ref{coro:abc+malle} in its present form cannot yield conclusions about covers with 3 branch points.

Moreover, by the Riemann-Hurwitz formula, the cover $f$ has at least 7 branch points, provided $X$ is of genus at least $2|G|-1$. Consequently, we have this conditional statement:

\vspace{2mm}

\noindent
{\it{The specialization set of a given regular $\Qq$-$G$-cover of $\mathbb{P}^1$ of genus at least $2|G|-1$ is of density $0$, under the abc-conjecture and the lower bound \eqref{conj:malle_lower}.}}
\end{remark}

In Corollary \ref{coro:abc_explicit} below, we give several explicit situations where the conclusion of Corollary \ref{coro:abc+malle} holds, independently of the ramification data of $f$:

\begin{corollary} \label{coro:abc_explicit}
Suppose the abc-conjecture holds and one of these conditions is satisfied:

\vspace{0.5mm}

\noindent
{\rm{(a)}} $G$ has rank at least 6 and \eqref{conj:malle_lower} holds for the group $G$,

\vspace{0.5mm}

\noindent
{\rm{(b)}} $G$ has a cyclic quotient of order $\not \in \{1,2,3,4,5,6,8,10,12\}$ and $G$ fulfills \eqref{conj:malle_lower},

\vspace{0.5mm}

\noindent
{\rm{(c)}} $G$ is nilpotent of order divisible by a prime number $\geq 7$.

\vspace{0.5mm}

\noindent
Then the density of ${\rm{Sp}}(f)$ is 0.
\end{corollary}

\begin{proof}
First, assume $G$ has rank $\geq 6$ and \eqref{conj:malle_lower} holds. Then, by the first condition and the {\it{Riemann existence theorem}}, $f$ has at least 7 branch points. Applying Corollary \ref{coro:abc+malle} and Remark \ref{rk:eq3} (with $S$ the set of all branch points of $f$) provides the desired conclusion. 

Now, assume $G$ has a cyclic quotient of order $\not \in \{1,2,3,4,5,6,8, 10, 12\}$ and $G$ fulfills \eqref{conj:malle_lower}. We shall make use of the following easy claim:

\vspace{2mm}

\noindent
{\it{Let $n$ be a positive integer $\not \in \{1,2,3,4,5,6,8, 10, 12\}$. Then every regular $\Qq$-$\Zz/n\Zz$-cover of $\mathbb{P}^1$ has either at least 8 branch points or at least 6 branch points of ramification index $\geq 3$.}}

\vspace{2mm}

\noindent
Under the claim, we may apply Corollary \ref{coro:abc+malle} and Remark \ref{rk:eq3} to get the desired conclusion.

We now prove the claim. If $p_0$ is a prime number and $m \geq 1$, recall that, as a classical con\-se\-quence of the {\it{Branch Cycle Lemma}} (see \cite{Fri77} and \cite[Lemma 2.8]{Vol96}), every regular $\Qq$-$\Zz/p_0^m\Zz$-cover of $\mathbb{P}^1$ has at least $p_0^{m} - p_0^{m-1}$ branch points of ramification index $p_0^m$ \footnote{Indeed, at least one such branch point must exist since the inertia groups at branch points generate $\Zz/p_0^m \Zz$. By the Branch Cycle Lemma, we obtain at least $\varphi(p_0^m) = p_0^{m} - p_0^{m-1}$ such branch points, where $\varphi$ denotes the Euler totient function. See, e.g., \cite[Proposition 3.1.19]{Deb09} for more details.}. Consequently, the claim already holds if $n$ is divisible by 16, 9, 25 or a prime number $p_0 \geq 7$. For the case $n=15$, let $g$ be a regular $\Qq$-$\Zz/15\Zz$-cover of $\mathbb{P}^1$. Then either $g$ has no branch point of ramification index 15, in which case $g$ has at least 6 branch points of ramification index $\geq 3$ (at least 2 coming from the subcover of degree 3 and at least 4 from that of degree 5), or $g$ has at least one branch point of ramification index $15$, in which case $g$ has in fact at least 8 such branch points by the Branch Cycle Lemma. In particular, the claim holds if $n$ is divisible by 15. As to the remaining cases $n=20, 24, 40$, one treats them as the case $n=15$.

Finally, (c) is a special case of (b). Indeed, if $G$ is nilpotent of order divisible by a prime $q$, then $G$ has a (cyclic) quotient group of order $q$, and $G$ fulfills \eqref{conj:malle_lower} by \cite{KM04}.
\end{proof}

\begin{remark} \label{rk:abc_explicit}
(a) More explicit examples derived from (b) could be given in (c). For example, the density zero conclusion also holds if $G$ is nilpotent of order divisible by 15. We refrain from considering more applications of this kind, to avoid complicated case distinctions.

\vspace{1mm}

\noindent
(b) By Corollary \ref{coro:abc_explicit}(c), if $q \geq 7$ is a prime number, then no regular $\Zz/q\Zz$-extension of $\Qq(T)$ is parametric, under the abc-conjecture. The interest of this remark is that none of the methods from \cite{KL18} and \cite[\S7]{KLN19} applies to finite groups of prime order. 

More generally, by the above, no regular $G$-extension of $\Qq(T)$ with $r \geq 7$ branch points is parametric, under the abc-conjecture and, possibly, the lower bound \eqref{conj:malle_lower}. In Appendix \ref{sec:parametric}, we discuss the situation where $r$ is 2 or 3. The case $r \in \{4,5,6\}$ remains open in general.
\end{remark}

\subsection{Variants} \label{ssec:uniform}

We provide below two variants of Corollary \ref{coro:abc+malle}. 

The first one asserts that one can remove the assumption that the lower bound \eqref{conj:malle_lower} holds, at the cost of making \eqref{eq2} less explicit:

\begin{theorem} \label{thm:compare}
There exists a positive constant $r_0(G)$ such that if $r \geq r_0(G)$ and if the abc-conjecture holds, then the set ${\rm{Sp}}(f)$ has density 0.
\end{theorem}

\begin{proof}
Without loss, we may assume $G$ is a regular Galois group over $\Qq$ \footnote{The definition of a regular Galois group over $\Qq$ is recalled in \S\ref{ssec:intro1}.}. Then, by \cite[Theorem 1.1]{Deb17}, there exists a positive constant $\beta(G)$ such that  the following holds for every sufficiently large $x$ (up to an explicit multiplicative constant depending on $G$): $$x^{\beta(G)} \leq | \overline{{\mathcal{S}}}(G,x)|.$$ 
Hence, by Theorem \ref{thm:abc_spec} and Remark \ref{rk:stup}(b), if $r \geq 5$, it suffices to check $e< \beta(G)$ (with $e$ as in \eqref{eq1.5}), which can be guaranteed if $r$ is sufficiently large (depending on $G$).
\end{proof}

\begin{remark}
In fact, \cite[Theorem 1.1 and \S4.1]{Deb17} provides the following:

\vspace{2mm}

\noindent
{\it{Let $f_1 : X_1\to \mathbb{P}^1$ be a regular $\qq$-$G$-cover with $r_1$ branch points. Then, for all sufficiently large $x$, one has $x^{a(G)/r_1} \le |{\rm{Sp}}(f_1) \cap \overline{\mathcal{S}}(G, x)|$, where $a(G)$ may be chosen as $(|G|-1) \cdot {|G|}^{-1} \cdot (3 |G|^4 \log(|G|))^{-1}$.}}

\vspace{2mm}

\noindent
Combination with our Theorem \ref{thm:abc_spec} then gives
$$ x^{a(G)/r} \le |{\rm{Sp}}(f) \cap \overline{\mathcal{S}}(G, x)| \le x^{b(G)/r},$$
where $b(G) > a(G)$ is an explicit positive constant depending only on $G$, under the abc-conjecture. In particular, if $f_2 : X_2\to \mathbb{P}^1$ is another regular $\qq$-$G$-cover with $r_2> ({b(G)}/{a(G)}) \cdot r_1$ branch points, then this inequality holds for every sufficiently large $x$, under the abc-conjecture:
$$|{\rm{Sp}}(f_2) \cap \overline{\mathcal{S}}(G, x)| < |{\rm{Sp}}(f_1) \cap \overline{\mathcal{S}}(G,x)|.$$

\noindent
As a consequence, the constant $r_0(G)$ in Theorem \ref{thm:compare} can be made explicit. Namely, if $G$ is not a regular Galois group over $\Qq$, one can arbitrarily take $r_0(G)=1$. Otherwise, take any $r_0(G) > ({b(G)}/{a(G)}) \cdot r_1(G)$, where $r_1(G)$ is the smallest number of branch points of a regular $\qq$-$G$-cover of $\mathbb{P}^1$.
\end{remark}

For our second variant, we need to recall beforehand the statements of the uniformity conjecture and the upper bound from the Malle conjecture.

\vspace{2mm}

\noindent
{\bf{The uniformity conjecture.}} {\it{Let $g \geq 2$ be an integer. Then there exists a positive integer $B$, depending only on $g$, such that the set of all $\Qq$-rational points on any given smooth curve defined over $\Qq$ with genus $g$ has cardinality at most $B$.}}

\vspace{2mm}

\noindent
{\bf{The Malle conjecture (upper bound).}} {\it{For every $\epsilon >0$, one has
\begin{equation} \label{conj:malle_upper}
|\mathcal{S}(G,x)| \leq c_2(G, \epsilon) \cdot x^{\alpha(G) + \epsilon}
\end{equation}
for some constant $c_2(G, \epsilon) >0$ and every sufficiently large $x$, where $\alpha(G)$ is defined in \eqref{eq:intro_3}.}}

\begin{theorem} \label{thm:uniform}
Suppose the uniformity conjecture holds and $G$ has a normal subgroup $H$ such that the following three conditions are satisfied:

\vspace{0.5mm}

\noindent
{\rm{(a)}} for every regular $\qq$-$G/H$-cover $X \rightarrow \mathbb{P}^1$, the genus of $X$ is at least 2,

\vspace{0.5mm}

\noindent
{\rm{(b)}} $p$ does not divide the order of $G/H$,

\vspace{0.5mm}

\noindent
{\rm{(c)}} \eqref{conj:malle_lower} and \eqref{conj:malle_upper} hold for the groups $G$ and $G/H$, respectively.

\vspace{0.5mm}

\noindent
Then the set ${\rm{Sp}}(f)$ has density 0.
\end{theorem}

\begin{remark} \label{rk:uniform}
(a) By \cite[Theorem 1.1]{CHM97}, the uniformity conjecture holds under the Lang conjecture, which asserts that the set of all $\Qq$-rational points on any variety of general type defined over $\Qq$ is not Zariski dense.

\vspace{0.5mm}

\noindent
(b)  By \cite[Proposition 7.3]{KL18}, Condition (a) of Theorem \ref{thm:uniform} holds if $G/H$ is neither solvable of even order nor of order 3.
\end{remark}

\begin{proof}
As noted in \S\ref{sssec:basics_1.1}, it suffices to show that the set ${\rm{Sp}}(E)$ has density zero.

Let $E_1/\Qq(T), \dots,$ $E_s/\Qq(T)$ be the subextensions of $E/\Qq(T)$ of group $G/H$. For $i \in \{1, \dots,s\}$, let $g_i$ be the genus of $X_i$, where $X_i \rightarrow \mathbb{P}^1$ is the regular $\Qq$-$G/H$-cover associated with $E_i/\Qq(T)$. Also, let $q$ be the smallest prime divisor of the order of $G/H$. One then has 
$$\alpha(G/H)= \frac{|G|^{-1}}{|H|^{-1}} \bigg(1- \frac{1}{q} \bigg)^{-1}.$$

Let $F/\Qq$ be a $G$-extension in ${\rm{Sp}}(E)$ and $t_0 \in \mathbb{P}^1(\Qq)$ such that $F={E}_{t_0}$. By \cite[Lemma 3.2]{KL18}, $(E_1)_{t_0}/\Qq, \dots, (E_s)_{t_0}/\Qq$ are the distinct subextensions of $E_{t_0}/\Qq$ with Galois group $G/H$. Hence, there exists $i \in \{1, \dots,s\}$ such that $F^H/\Qq$ is the specialization of $E_i/\Qq(T)$ at $t_0$. Let $g_0 = {\rm{max}}(g_1, \dots, g_s)$. By (a) and as the uniformity conjecture holds, one may apply \cite[Proposition 2.5]{KL18} to get that there exists a positive constant $B=B(|G / H|, g_0)$ such that, for each $i \in \{1, \dots,s\}$, there exist at most $B$ points $t_0 \in \mathbb{P}^1(k)$ with $F^H/\Qq = (E_i)_{t_0}/\Qq$. Moreover, if $d_F$ and $d_{F^H}$ denote the absolute discriminants of the number fields $F$ and $F^H$, respectively, then one has $|d_{F^H}| \leq |d_F|^{1/|H|}$. Conclude that this inequality holds for every positive integer $x$:
\begin{equation} \label{eq:G/H}
|{\rm{Sp}}(E) \cap \mathcal{S}(G,x)| \leq Bs \cdot |\mathcal{S}(G/H,x^{1/|H|})|.
\end{equation}
By (b), one has $p < q$, that is, $(1/|H|) \cdot \alpha(G/H) < \alpha(G)$. Let $\epsilon >0$ be such that 
\begin{equation} \label{expo}
{\alpha(G/H) + \epsilon} < \alpha(G) \cdot |H|.
\end{equation}
Combining \eqref{eq:G/H} and the assumption that \eqref{conj:malle_upper} holds for the group $G/H$ then provides
\begin{equation} \label{eq:G/H2}
|{\rm{Sp}}(E) \cap \mathcal{S}(G,x)| \leq Bs \cdot c_2(G/H, \epsilon) \cdot x^{(\alpha(G/H) + \epsilon) / |H|}
\end{equation}
for some positive constant $c_2(G/H, \epsilon)$ and every $x \geq x(G/H, \epsilon)$. On the other hand, since \eqref{conj:malle_lower} has been assumed to hold for the group $G$, one has
\begin{equation} \label{eqref:G}
|\mathcal{S}(G,x)| \geq c_1(G) \cdot x^{\alpha(G)}
\end{equation}
for some positive constant $c_1(G)$ and every $x \geq x(G, \epsilon)$. Combine \eqref{eq:G/H2} and \eqref{eqref:G} to get
\begin{equation}\label{eq:b}
\frac{|{\rm{Sp}}(E) \cap \mathcal{S}(G,x)|}{|\mathcal{S}(G,x)|} = O(x^{(\alpha(G/H) + \epsilon)/|H|- \alpha(G)}), \quad x \rightarrow \infty.
\end{equation}
It then remains to combine \eqref{expo} and \eqref{eq:b} to conclude that the set ${\rm{Sp}}(E)$ has density 0.
\end{proof}

\begin{corollary} \label{coro:uniform}
Suppose the uniformity conjecture holds, the group $G$ is nilpotent, and one of the following two conditions is satisfied:

\vspace{0.5mm}

\noindent
{\rm{(a)}} $G$ is of even order but $|G| \not \in \{2^a3^b \, : \, a \geq 1, \, b \in \{0,1\}\}$,

\vspace{0.5mm}

\noindent
{\rm{(b)}} $G$ is of odd order and $|G|$ has at least two distinct prime factors.

\vspace{0.5mm}

\noindent
Then the set ${\rm{Sp}}(f)$ has density 0.
\end{corollary}

For example, Corollary \ref{coro:uniform} applies to the groups $\Zz/10\Zz$ and $\Zz/3\Zz \times \Zz/6\Zz$. Note that these groups have covers with four branch points and our results under the abc-conjecture cannot (a priori) apply to them.

\begin{proof}
As the group $G$ is nilpotent, \cite{KM04} may be used to show that the entire Malle conjecture holds for every quotient of $G$. By Theorem \ref{thm:uniform}, it then suffices to find a quotient of $G$ for which Conditions (a) and (b) of that theorem hold.

Set $G= P_1 \times \cdots \times P_s$ where $s \geq 1$ and $P_i$ is a non-trivial $p_i$-group for each $i \in \{1, \dots,s\}$. We assume $p_1 \leq \cdots \leq p_s$ and $|P_i| \leq |P_j|$ if $p_i=p_j$ ($i,j \in \{1, \dots, s\}$). If (a) holds, then $p_s \geq 5$ or ($p_s = 3$ and $|P_s| \geq 9$) or ($p_s=3$ and $P_s \times P_{s-1} \cong \Zz/3\Zz \times \Zz/3\Zz$). Then either $G/ (P_1 \times \cdots \times P_{s-1})$ (in the first two cases) or $G/ (P_1 \times \cdots \times P_{s-2})$ (in the third case) has odd order and it is not $\Zz/3\Zz$. In particular, Conditions (a) and (b) of Theorem \ref{thm:uniform} are fulfilled (see Remark \ref{rk:uniform}(b)). If (b) holds, then $p_1 < p_s$ and $p_s \geq 5$, and one concludes as in (a).
\end{proof}

\subsection{Proof of Theorem \ref{thm:abc_spec}} \label{ssec:proof}

The proof relies on this consequence of the abc-conjecture, due to Elkies, Langevin, and Granville (see, e.g., \cite[Theorem 5]{Gra98}):

\begin{theorem} \label{langevin}
Let $P(U,V) \in \zz[U,V]$ be a homogeneous polynomial of degree $d$ without any repeated factors. Assume the abc-conjecture holds. Then, for every $\epsilon>0$ and every couple $(u,v)$ of coprime integers, one has
$${\rm{rad}}(P(u,v)) \geq c_1 \cdot \max\{|u|, |v|\}^{d-2-\epsilon},$$
where $c_1$ is a positive constant depending only on $P$ and $\epsilon$.
\end{theorem}

We break the proof of Theorem \ref{thm:abc_spec} into three parts.

\subsubsection{Controlling the ramification of specializations of $f$}

The first part requires associating a homogeneous polynomial controlling the ramification behaviour in specializations of $f$, which is done via Theorem \ref{beckmann}. 

For each $i\in \{1, \dots ,r\}$, let $P_i(U,V)\in \zz[U,V]$ be the minimal polynomial of the branch point $t_i$. Set $$P(U,V) = \prod_{i \in I} P_i(U,V),$$
where the $t_i$'s, $i \in I$, build a set of representatives of $t_1, \dots, t_r$ modulo the action of G$_\Qq$. Moreover, set $a_i=|G|(1-1/e_i)$, where $e_i$ is as before the ramification index of $t_i$, for each $i \in I$ (so $a_i$ is the index of an inertia group generator at $t_i$, viewed as a permutation in the regular permutation action of $G$ \footnote{Recall that the {\it{index}} of a permutation $\sigma \in S_d$ is defined as $d$ minus the number of orbits of $\langle \sigma\rangle$.}). For $t_0 \in \Qq$, set $t_0 = u/v$, with $u$ and $v$ coprime integers, and denote the absolute discriminant of $E_{t_0}$ by $d_{t_0}$.

Let $\ell$ be a prime number, not contained in the finite exceptional set $\mathcal{S}_{\rm{exc}}$ from Theorem \ref{beckmann}. By that theorem, $\ell$ is (tamely) ramified in $E_{t_0}/\Qq$ with ramification index $e_i$ if $\ell$ divides $P_i(u,v)$ with positive multiplicity at most $q_i-1$, where $q_i$ is the smallest prime divisor of $e_i$. In that case, the exponent of $\ell$ in $d_{t_0}$ equals $a_i$. Therefore, we get the following lower bound:

$$ |d_{t_0}| \geq \prod_{i \in I} \prod_\ell \ell^{|G|(1-1/e_i)},$$
where, given $i \in I$, the second product is over all prime numbers $\ell$ which are not in $\mathcal{S}_{\rm{exc}}$ and which divide $P_i(u,v)$ with positive multiplicity at most $q_i-1$. As the finitely many elements of the set $\mathcal{S}_{\rm{exc}}$, as well as the numbers $q_i$, $i \in I$, are fixed and depend only on $f$, we have 
\begin{equation} \label{eq:disc1}
|d_{t_0}| \geq c_0 \cdot \prod_{i \in I} \prod_\ell \ell^{|G|(1-1/e_i)},
\end{equation}
for some positive constant $c_0$ depending only on $f$, and where, given $i \in I$, the second product is over {\textit{all}} prime numbers $\ell$ which divide $P_i(u,v)$ with positive multiplicity at most $q_i-1$. Combining \eqref{eq:disc1} and the definitions of $e_0$ and $q_0$ (see the beginning of \S\ref{sec:densityI}) yields the following lower bound:
\begin{equation} \label{eq:disc2}
|d_{t_0}| \geq  c_0 \cdot \prod_\ell \ell^{|G|(1-1/e_0)},
\end{equation}
where the product is over all primes dividing $P(u,v)$ with positive multiplicity at most $q_0-1$.

\subsubsection{Applying Theorem \ref{langevin}}

The second part consists in estimating the product of all prime numbers dividing a given value of $P(U,V)$ with positive multiplicity at most $q_0-1$.

Let $u, \, v$ be coprime integers and set $n = {\rm{max}} \{|u|, |v|\}$. Given $\epsilon >0$, since the abc-conjecture has been assumed to hold, we may apply Theorem \ref{langevin} to get this lower bound:
\begin{equation} \label{eq:disc2.5}
{\textrm{rad}}(P(u,v)) \geq c_1 \cdot n^{\deg(P)-2-\epsilon} = c_1 \cdot n^{r-2-\epsilon},
\end{equation}
where $c_1$ depends only on $P(U,V)$ and $\epsilon$. For $m \geq 1$, let $B_m$ be the product of all prime numbers dividing $P(u,v)$ exactly $m$ times. Setting $t_0=u/v$, \eqref{eq:disc2} can be rewritten as
\begin{equation} \label{eq:disc3}
|d_{t_0}| \geq c_0 \cdot \bigg(\prod_{m=1}^{q_0-1} B_m \bigg)^{|G|(1-1/e_0)}.
\end{equation}

Now, let $B_{\geq q_0}$ be the product of all $B_m$'s with $m \geq q_0$. Since ${\textrm{rad}}(P(u,v))$ is the product of all $B_m$'s with $m \geq 1$, one has
\begin{equation} \label{eq:disc4}
{\textrm{rad}}(P(u,v)) \leq \frac{|P(u,v)|}{\displaystyle{\prod_{m=q_0}^\infty} B_m^{m-1}} \leq \frac{|P(u,v)|}{B_{ \geq q_0}^{q_0-1}}.
\end{equation}
As $|P(u,v)| \leq  c_2 \cdot n^r$, with $c_2=c_2(P)$, the combination of \eqref{eq:disc2.5} and \eqref{eq:disc4} then yields
$$\frac{c_2 \cdot n^r}{B_{ \geq q_0}^{q_0-1}} \geq c_1 \cdot n^{r-2-\epsilon},$$
that is,
\begin{equation} \label{eq:disc5}
B_{ \geq q_0} \leq c_3 \cdot n^{(2+ \epsilon)/(q_0-1)}
\end{equation}
for some positive constant $c_3$ depending only on $f$ and $\epsilon$. Combining \eqref{eq:disc2.5}, \eqref{eq:disc3}, and \eqref{eq:disc5} then provides the following bound (up to replacing $\epsilon$ by $\epsilon|G|^{-1}(1-1/e_0)^{-1}(q_0-1)/q_0$):
\begin{equation} \label{eq:disc6}
|d_{t_0}|\geq c_0 \cdot \bigg(\frac{{\textrm{rad}}(P(u,v))}{B_{ \geq q_0}} \bigg)^{|G|(1-1/e_0)} \geq c_4 \cdot n^{|G|(1-1/e_0)(r-2-2/(q_0-1))-\epsilon},
\end{equation}
where $c_4$ is some positive constant depending only on $f$ and $\epsilon$.

\subsubsection{Conclusion}

Finally, we use the estimate \eqref{eq:disc6} to bound $|{\rm{Sp}}(f) \cap \overline{\mathcal{S}}(G,x)|$ from above. 

Let $\delta$ be any real number such that
$$\delta >\frac{r}{r - 2-2/(q_0-1)}.$$
By \eqref{eq1}, $\delta$ is well-defined and positive. Let $\epsilon$ be a positive real number. By \eqref{eq:disc6}, for every couple $(u,v)$ of coprime integers, one has
\begin{equation} \label{eq:disc7}
|d_{u/v}|  \geq c_4 \cdot {\rm{max}} \{|u|, |v|\}^{r \cdot |G|(1-1/e_0) \cdot \delta^{-1}-\epsilon}.
\end{equation}
Let $n$ be a sufficiently large integer (depending on $\epsilon$). The lower bound \eqref{eq:disc7} then allows to conclude that all specializations $E_{t_0}/\Qq$ of $E/\Qq(T)$ with $t_0 \in \Qq$ and such that $$|d_{t_0}| \leq  n^{r \cdot |G|(1-1/e_0) \cdot \delta^{-1}}$$ must come from values $t_0 = u/v$ with ${\rm{max}} \{|u|, |v| \} \leq n$. Setting $x=n^{r \cdot |G| (1-1/e_0)\cdot \delta^{-1}}$, we find that all specializations $E_{t_0}/\Qq$ of $E/\Qq(T)$ with $t_0 \in \Qq$ and such that $|d_{t_0}| \leq x$ must come from values $t_0 = u/v$ with $${\rm{max}} \{|u|, |v| \} \leq x^{\delta\cdot (r |G| (1-1/e_0))^{-1}}.$$ In particular, choosing 
$$\delta = \frac{r}{r - 2-2/(q_0-1)} + \frac{\epsilon}{2} \cdot r \cdot |G| \cdot \bigg( 1 - \frac{1}{e_0} \bigg)$$
and using the definition of our exponent $e$ given in \eqref{eq1.5}, we obtain
$${\rm{max}} \{|u|, |v| \} \leq x^{(e + \epsilon)/2}.$$
As there are at most $4 \cdot x^{e + \epsilon}$ such pairs of integers $(u,v)$, this concludes the proof.

\section{Unconditional results} \label{sec:densityII}

The aim of this section is to show unconditionally that the set of specializations of almost all regular $\Qq$-$G$-covers of $\mathbb{P}^1$, where $G=\Zz/2\Zz$ or $S_3$, has density zero.

\subsection{The quadratic case} \label{ssec:quadraticII}

We start with the case $G=\Zz/2\Zz$ and, for simplicity, use the function field extension language, which is strictly identical to the cover point of view. 

\subsubsection{Main result}

The following statement is a more precise version of Theorem \ref{thm:intro_3} from the introduction. Note that the unconditional upper bound in (b) is expectedly weaker than the one pro\-vided by Theorem \ref{thm:abc_spec} under the abc-conjecture. Recall that the sets $\mathcal{E}(r)$ and $\mathcal{E}(r,H)$, which occur in the statement below, are defined in \S\ref{ssec:basics_3}.

\begin{theorem} \label{thm:Z/2Z_even}
Let $r$ be an even positive integer. Then there exists a subset $S$ of $\mathcal{E}(r)$ which satisfies the following two conclusions.

\vspace{0.5mm}

\noindent
{\rm{(a)}} One has $$\frac{|S \cap \mathcal{E}(r,H)|}{|\mathcal{E}(r,H)|} = 1-O \bigg( \frac{{\log}(H)}{\sqrt{H}} \bigg), \quad H \rightarrow \infty.$$
In particular, the set $S$ has density 1.

\vspace{0.5mm}

\noindent
{\rm{(b)}} For every extension $E/\Qq(T)$ in $S$, there exists a positive constant $\alpha < 1$ such that
$${|{\rm{Sp}}(E) \cap \mathcal{S}(\Zz/2\Zz,x)|} = O(x \cdot {\log}^{- \alpha}(x)), \quad x \rightarrow \infty.$$
In particular, the set of specializations of every extension of $\Qq(T)$ in $S$ has density 0.
\end{theorem}

\subsubsection{Proof of Theorem \ref{thm:Z/2Z_even}}

Our main tool is the case $n=2$ of the following diophantine result, which has its own interest and which shows that almost all twists of almost all superelliptic curves over $\Qq$ have only trivial $\Qq$-rational points, under a suitable assumption on the degree. Recall that the sets $\mathcal{P}(n,N)$ and $\mathcal{P}(n,N,H)$, and the sets $\mathcal{N}_n(P(T))$ and $\mathcal{N}_n(P(T),x)$, which occur in the statement below, are defined in \S\ref{ssec:basics_0} and \S\ref{ssec:basics_2}.3, respectively.

\begin{theorem} \label{thm:super_uncond}
Given two positive integers $n$ and $N$ such that $2 \leq n$ and $n$ divides $N$, there exists a subset $S'$ of $\mathcal{P}(n,N)$ such that the following two conclusions are satisfied.

\vspace{0.5mm}

\noindent
{\rm{(a)}} One has
\begin{equation} \label{eq:bound0}
\frac{|{S}' \cap \mathcal{P}(n,N,H)|}{|\mathcal{P}(n,N,H)|} = 1 - O \bigg(\frac{{\log}(H)}{\sqrt{H}} \bigg), \quad H \rightarrow \infty.
\end{equation}
In particular, the set ${S}'$ has density 1.

\vspace{0.5mm}

\noindent
{\rm{(b)}} For each $P(T) \in S'$, there exists a positive constant $\alpha < 1$ such that 
\begin{equation} \label{eq:bound}
|\mathcal{N}_n(P(T),x)|= O(x \cdot {\log}^{- \alpha}(x)), \quad x \rightarrow \infty.
\end{equation}
In particular, for each $P(T) \in S'$, the density of the subset $\mathcal{N}_n(P(T))$ of $\mathcal{N}_n$ is 0.
\end{theorem}

\begin{proof}[Proof of Theorem \ref{thm:Z/2Z_even} under Theorem \ref{thm:super_uncond}]
Let $S'$ be a subset of $\mathcal{P}(2,r)$ as in Theorem \ref{thm:super_uncond} and $\overline{S'} = \mathcal{P}(2,r) \setminus S'$. Let $S$ be the subset of $\mathcal{E}(r)$ consisting of all regular $\Zz/2\Zz$-extensions of $\Qq(T)$ with $r$ branch points and whose associated polynomial lies in the set ${S'} \cap \mathcal{P}_2(2,r)$ \footnote{The set $\mathcal{P}_2(2,r)$ is defined in \S\ref{ssec:basics_0}.}. 

First, we prove (a). For every positive integer $H$, one has
$$\begin{array}{lll}
1- \frac{|S \cap \mathcal{E}(r,H)|}{|\mathcal{E}(r,H)|} &= & 1- \frac{|S' \cap \mathcal{P}_2(2,r,H)|}{|\mathcal{E}(r,H)|}\\
& = & 1  - \frac{|\mathcal{P}_2(2,r, H)|}{|\mathcal{E}(r,H)|}  +\frac{|\overline{S'} \cap \mathcal{P}_2(2,r, H)|}{|\mathcal{E}(r,H)|}\\
& \leq & 1  - \frac{|\mathcal{P}_2(2,r, H)|}{|\mathcal{E}(r,H)|}  +\frac{|\overline{S'} \cap \mathcal{P}(2,r, H)|}{|\mathcal{E}(r,H)|}\\
& = & 1  - \frac{|\mathcal{P}_2(2,r, H)|}{|\mathcal{E}(r,H)|}  +\frac{|\mathcal{P}(2,r, H)|}{|\mathcal{E}(r,H)|} - \frac{|{S'} \cap \mathcal{P}(2,r, H)|}{|\mathcal{E}(r,H)|}\\
& = & 1  - \frac{|\mathcal{P}_2(2,r, H)|}{|\mathcal{E}(r,H)|}  +  \frac{|\mathcal{P}(2,r, H)| - |{S'} \cap \mathcal{P}(2,r, H)|}{|\mathcal{P}(2,r, H)|} \cdot \frac{|\mathcal{P}(2,r, H)|}{|\mathcal{E}(r,H)|}.\\
\end{array}$$
By Theorem \ref{thm:super_uncond}(a), one has
$$\frac{|\mathcal{P}(2,r, H)| - |{S'} \cap \mathcal{P}(2,r, H)|}{|\mathcal{P}(2,r, H)|} = O \bigg( \frac{{\log}(H)}{\sqrt{H}} \bigg), \quad H \rightarrow \infty.$$
Moreover, by Proposition \ref{prop:card} and as $|\mathcal{P}(2,r, H)| \sim 2^{r+1} \cdot H^{r+1}$ as $H$ tends to $\infty$, one has
$$1  - \frac{|\mathcal{P}_2(2,r, H)|}{|\mathcal{E}(r,H)|} = O \bigg( \frac{1}{H} \bigg) \,\, \, \, \, {\rm{and}} \,\, \, \, \, 
\frac{|\mathcal{P}(2,r, H)|}{|\mathcal{E}(r,H)|} = O(1)$$
as $H$ tends to $\infty$. Hence, one has
$$1- \frac{|S \cap \mathcal{E}(r,H)|}{|\mathcal{E}(r,H)|} = O \bigg( \frac{1}{H} \bigg) + O \bigg( \frac{{\log}(H)}{\sqrt{H}} \bigg) \cdot O(1) = O \bigg( \frac{{\log}(H)}{\sqrt{H}} \bigg), \quad H \rightarrow \infty.$$

Now, we prove (b). Given $E/\Qq(T) \in S$, there is a unique polynomial $P_E(T)$ in $S'$ with 
$$E=\Qq(T)(\sqrt{P_E(T)}).$$
By Theorem \ref{thm:super_uncond}(b), there is a constant $\alpha \in ]0,1[$ with
$$|\mathcal{N}_2(P_E(T), x)| = O(x \cdot {\log}^{- \alpha}(x))$$
as $x$ tends to $\infty$. By applying Proposition \ref{prop:Z/2Z}(b), we get that $|\mathcal{N}_2(P_E(T), x)|$ is the cardinality of the subset of $\mathcal{N}_2(x)$ (see \S\ref{ssec:basics_0}) defined by the extra condition that $\Qq(\sqrt{d})/\Qq$ is in ${\rm{Sp}}(E)$. As the absolute discriminant of the number field $\Qq(\sqrt{d})$ is $d$ or $4d$ ($d \in \mathcal{N}_2$), we get 
$$|{\rm{Sp}}(E) \cap \mathcal{S}(\Zz/2\Zz,x)| \leq |\mathcal{N}_2(P_E(T), x)| = O(x \cdot {\log}^{- \alpha}(x))$$
as $x$ tends to $\infty$. It then remains to use that $$|\mathcal{S}(\Zz/2\Zz,x)| \geq |\mathcal{N}_2(x/4)| \sim \frac{x}{4 \cdot \zeta(2)}$$
\noindent
as $x$ tends to $\infty$ to get the desired density zero conclusion.
\end{proof}

\begin{proof}[Comments on proof of Theorem \ref{thm:super_uncond}]
The proof is similar to the arguments given in \cite[\S3.2 and \S4.2]{Leg18b}, which yield Theorem \ref{thm:super_uncond} with the weaker conclusion that almost all superelliptic curves over $\Qq$ have at least one twist with only trivial $\Qq$-rational points. For the convenience of the reader, we offer in Appendix \ref{app:proof_1} a full proof of Theorem \ref{thm:super_uncond} with the necessary adjustments to get the desired stronger conclusion.
\end{proof}

In Appendix \ref{app:proof_2}, we give two variants of Theorem \ref{thm:super_uncond} where we relax the assumption that $n$ divides $N$, at the cost of making the conclusion in (b) weaker.

\subsection{Symmetric groups} \label{sec:S3}

The aim of this subsection is to give evidence that, given $n \geq 2$, almost all regular $\Qq$-$S_n$-covers of $\mathbb{P}^1$ have a specialization set of density $0$, thus generalizing the conclusion of Theorem \ref{thm:Z/2Z_even}. We count those covers via degree $n$ polynomials with Galois group $S_n$ over $\qq(T)$. For $n=3$, we obtain an unconditional result, given in Theorem \ref{thm:S3}.

\subsubsection{Preliminaries} \label{ssec:prelim_S3}

First, we explain our way of counting covers via polynomials. Given $n \geq 2$, if $E/\Qq(T)$ denotes the function field extension associated with a regular $\qq$-$S_n$-cover of $\mathbb{P}^1$, then $E$ is the splitting field over $\Qq(T)$ of a degree $n$ polynomial $$Y^n + a_{n-1}(T) Y^{n-1} + \cdots + a_1(T) Y + a_0(T),$$ with $a_{0}(T), \dots, a_{n-1}(T) \in \zz[T]$. A natural way of counting covers is then to count the corresponding polynomials up to a bounded $T$-degree and bounded height.

Given $n \geq 2$ and $D \geq 1$, we therefore consider the set $\mathcal{Q}(n,D)$ of all polynomials $P(T,Y)\in \zz[T][Y]$ which are monic and of degree $n$ in $Y$, and which are also of degree at most $D$ in $T$. Given $i \in \{0, \dots, n-1\}$ and $j \in \{0, \dots, D\}$, let $a_{i,j}\in \zz$ denote the coefficient at $T^j$ of $a_i(T)$. We then count covers by fixing an integer $H \geq 1$ and considering the set
$$\mathcal{Q}(n,D,H) = \{P(T,Y) \in \mathcal{Q}(n,D) \ : \ |a_{i,j}| \leq H \text{ for all } i,j\}.$$ 

\subsubsection{Main result} \label{ssec:thm_S3}

Our eventual goal is to prove Theorem \ref{thm:S3} below, which is a statement about Galois covers with group $S_3$. Since most of the ingredients in the proof are not specific to the case $n=3$, we try to retain generality as long as possible.

\begin{lemma} \label{lemma_0}
Given $n \geq 2$ and $D \geq 1$, let $U_{n-1, D}, \dots, U_{n-1,0}, \dots, U_{1, D}, \dots, U_{1,0}, U_{0, D}, \dots, U_{0,0}$ be algebraically independent indeterminates, and denote by $\underline{U}$ the vector consisting of these variables.
Let $\Delta(T)\in \qq[\underline{U}][T]$ be the discriminant (with respect to $Y$) of the polynomial
$$F(\underline{U},T,Y)=Y^n + \bigg(\sum_{j=0}^D U_{n-1, j} T^j \bigg) Y^{n-1} + \cdots + \bigg(\sum_{j=0}^D U_{1,j} T^j \bigg) Y + \sum_{j=0}^D U_{0,j} T^j.$$
Then $\Delta(T)$ is irreducible over $\Qq(\underline{U})$.
\end{lemma}

\begin{proof}
It is well-known that the discriminant of the polynomial
$$Y^n + U_{n-1,0} Y^{n-1} + \cdots + U_{1,0} Y + U_{0,0}$$ is irreducible as an element of $\qq[U_{0,0}, U_{1,0}, \dots,U_{n-1,0}]$ (see, e.g., \cite[page 15]{GKZ94}). The polynomial $F(\underline{U}, T, Y)$ arises from this polynomial after applying the map sending 
$U_{i,0}$
to 
$$U_{i,0} + (U_{i, 1} T + U_{i,2} T^2 +  \cdots + U_{i,D} T^D)$$ 
for each $i \in \{0, \dots, n-1\}$ and fixing all other indeterminates. Since this corresponds to an automorphism of the ring $\qq[\underline{U}, T, Y]$,
the discriminant $\Delta(T)$ must still be irreducible as an element of $\qq[\underline{U}, T]$, and hence also inside $\qq(\underline{U})[T]$ by Gauss' lemma.
\end{proof}

\begin{lemma} \label{lemma_0.5}
Given $n \geq 2$ and $D \geq 1$, consider the set $S$ consisting of all polynomials 
$$P(T,Y) = Y^n + a_{n-1}(T) Y^{n-1} + \cdots + a_1(T) Y + a_0(T)$$
in $\mathcal{Q}(n,D)$ fulfilling the following three conditions:

\vspace{0.5mm}

\noindent
{\rm{(a)}} $P(T,Y)$ has Galois group $S_n$ over $\Qq(T)$,

\vspace{0.5mm}

\noindent
{\rm{(b)}} the discriminant $\Delta(T) \in \Zz[T]$ of $P(T,Y)$ is irreducible,

\vspace{0.5mm}

\noindent
{\rm{(c)}} the polynomial 
$$a_{n-1,D} Y^{n-1} + a_{n-2,D} Y^{n-2} + \cdots + a_{0,D}$$
has degree $n-1$ and is irreducible over $\Qq$, where $a_{i,D}$ denotes the coefficient of $a_i(T)$ at $T^D$ for each $i \in \{0, \dots, n-1\}$.

\vspace{0.5mm}

\noindent
Then one has
$$\frac{|S \cap \mathcal{Q}(n,D,H)|}{|\mathcal{Q}(n,D,H)|} = 1 - O \bigg( \frac{{\log}(H)}{\sqrt{H}} \bigg), \quad H \rightarrow \infty.$$
In particular, the set $S$ has density 1.
\end{lemma}

\begin{proof}
We estimate the size of the complement $\mathcal{Q}(n,D) \setminus S$. Let $S_{(1)}$ (resp., $S_{(2)}$, $S_{(3)}$) be the subset of $\mathcal{Q}(n,D)$ which consists of all polynomials $P(T,Y)$ which do not fulfill (a) (resp., (b), (c)). It is enough to show that
\begin{equation} \label{eq:set_gen}
\frac{|S_{(j)} \cap \mathcal{Q}(n,D,H)|}{|\mathcal{Q}(n,D,H)|} = O \bigg( \frac{{\log}(H)}{\sqrt{H}} \bigg), \quad H \rightarrow \infty,
\end{equation}
for each $j \in \{1, 2, 3\}$. This is mainly obtained by making use of a sufficiently precise version of Hilbert's irreducibility theorem (namely, \cite[Theorem 2.1]{Coh81b}).

Given algebraically independent indeterminates $T_0, \dots, T_{n-1}$, the polynomial
$$Y^n + T_{n-1} Y^{n-1} + \cdots + T_1 Y + T_0$$
has Galois group $S_n$ over $\Qq(T_0, T_1, \dots, T_{n-1})$. Apply \cite[Theorem 2.1]{Coh81b} to get that the number of tuples $(t_0, t_1, \dots, t_{n-1})$ of integers of absolute value at most $H$ such that 
$$Y^n + t_{n-1} Y^{n-1} + \cdots + t_1 Y + t_0$$
does not have Galois group $S_n$ over $\Qq$ is $O(H^{n-1/2} \cdot {\rm{log}}(H))$ as $H$ tends to $\infty$. Combine this and the fact that if $P(T,Y)$ is such that $P(0,Y)$ has Galois group $S_n$ over $\Qq$, then $P(T,Y)$ has Galois group $S_n$ over $\Qq(T)$ to get that \eqref{eq:set_gen} holds for $j=1$. In the same way, \eqref{eq:set_gen} also holds if $j=2$ (use Lemma \ref{lemma_0}), and if $j=3$.
\end{proof}

\begin{lemma} \label{lemma_1}
Let $n \geq 2$ and $D \geq 1$ be integers. Let $S'$ be the subset of $\mathcal{Q}(n,D)$ which consists of all polynomials $P(T,Y)$ fulfilling the following two conditions:

\vspace{0.5mm}

\noindent
{\rm{(a)}} $P(T,Y)$ defines a regular $\Qq$-$S_n$-cover $f: X \rightarrow \mathbb{P}^1$ of branch points $t_1, \dots, t_r$,

\vspace{0.5mm}

\noindent
{\rm{(b)}} $t_1,\dots, t_r$ are algebraically conjugate.

\vspace{0.5mm}

\noindent
Then one has
$$\frac{|S' \cap \mathcal{Q}(n,D,H)|}{|\mathcal{Q}(n,D,H)|} = 1 - O \bigg( \frac{{\log}(H)}{\sqrt{H}} \bigg), \quad H \rightarrow \infty.$$
In particular, the set $S'$ has density 1.
\end{lemma}

\begin{proof}
It suffices to show that the set $S$ provided by Lemma \ref{lemma_0.5} is a subset of $S'$. Let $P(T,Y)$ be in $S$ and $\Delta(T) \in \Zz[T]$ its discriminant. By Condition (b) of Lemma \ref{lemma_0.5}, $\Delta(T)$ cannot be a square in $\overline{\Qq}(T)$, i.e., the Galois group $\overline{G}$ of $P(T,Y)$ over $\overline{\Qq}(T)$ is not contained in $A_n$. Since $\overline{G} \trianglelefteq S_n$ (by Condition (a) of Lemma \ref{lemma_0.5}), we can conclude that $\overline{G} =S_n$, thus leading to (a). As for (b), it suffices to show that $\infty$ is not a branch point of $f$ (by the second part of Condition (b) of Lemma \ref{lemma_0.5}), since every finite branch point of $f$ is a root of $\Delta(T)$. 

Setting $U=1/T$, consider the polynomial
$$Q(U,Y)=U^D P(1/U,Y) =U^D Y^n + b_{n-1}(U) Y^{n-1} + \cdots + b_1(U) Y + b_0(U),$$
where $b_i(U) = a_i(1/U) U^D$, $i \in \{0, \dots, n-1\}$. Since $Q(0,Y)$ is separable (even irreducible) of degree $n-1$ (by Condition (c) of Lemma \ref{lemma_0.5}), $U=0$ has $n$ distinct preimages under
the degree $n$ regular $\Qq$-cover of $\mathbb{P}^1$ defined by $P(T,Y)$ (namely, $n-1$ distinct points with finite $Y$-coordinate, and one more infinite point). It is therefore unramified at $U=0$, as is its Galois closure $f$. This concludes the proof.
\end{proof}

\begin{lemma} \label{lemma_2}
Let $E/\Qq(T)$ be a regular $G$-extension all of whose branch points are algebraically conjugate. Then there exists a positive density set $\mathcal{S}_0$ of prime numbers such that all specializations of $E/\Qq(T)$ are unramified at all the primes in $\mathcal{S}_0$.
\end{lemma}

\begin{proof}
Let $R(T) \in \Qq[T]$ be the minimal polynomial over $\Qq$ of the branch points of $E/\Qq(T)$, $F$ the splitting field of $R(T)$ over $\Qq$, and $G={\rm{Gal}}(F/\Qq)$. Then $G$ is transitive, and so there exists an element $\sigma$ of $G$ fixing no branch point of $E/\Qq(T)$. Let $\mathcal{S}_0$ denote the set of all prime numbers $p$ such that the Frobenius associated with $p$ in $F/\Qq$ is conjugate in $G$ to $\sigma$. By the Chebotarev density theorem, $\mathcal{S}_0$ has density $\alpha=|C_\sigma|/|G| \in ]0,1[$, with $C_\sigma$ the conjugacy class of $\sigma$ in $G$. Moreover, by the definition of $\mathcal{S}_0$, no prime number $p \in \mathcal{S}_0$ (possibly up to finitely many exceptions) is a {\it{prime divisor}} of $R(T)$, that is, there exist no $t_0 \in \Qq$ such that $v_p(R(t_0)) >0$. Theorem \ref{beckmann} then yields that, for every prime number $p \in \mathcal{S}_0$ (possibly up to finitely many exceptions), no specialization of $E/\Qq(T)$ ramifies at $p$.
\end{proof}

A ``moral" implication of Lemma \ref{lemma_2} is that, for covers $f$ as in Lemma \ref{lemma_1}, the set ${\textrm{Sp}}(f)$ cannot be too large. Turning this into a precise statement depends on precise knowledge about the distribution of $S_n$-extensions of $\Qq$, which, in general, is a very difficult problem. For the special case $n=3$, however, we have the following result:

\begin{theorem} \label{thm:S3}
Given a positive integer $D$, consider the set $S$ of all polynomials $P(T,Y) \in \mathcal{Q}(3,D)$ fulfilling the following two conditions:

\vspace{0.5mm}

\noindent
{\rm{(a)}} $P(T,Y)$ defines a regular $\qq$-$S_3$-cover $f : X \rightarrow \mathbb{P}^1$,

\vspace{0.5mm}

\noindent
{\rm{(b)}} there exists a positive constant $\alpha$ such that
$$\frac{|{\rm{Sp}}(f) \cap \mathcal{\overline{S}}(S_3,x)|}{|\mathcal{\overline{S}}(S_3,x)|}=O({\log}^{-\alpha}(x))$$
as $x$ tends to $\infty$ (in particular, the set ${\rm{Sp}}(f)$ has density 0).

\vspace{0.5mm}

\noindent
Then one has
$$\frac{|S \cap \mathcal{Q}(3,D,H)|}{|\mathcal{Q}(3,D,H)|} = 1 - O \bigg( \frac{{\log}(H)}{\sqrt{H}} \bigg)$$
as $H$ tends to $\infty$. In particular, the set $S$ has density 1.
\end{theorem}

\begin{proof}
We choose $S'$ as in Lemma \ref{lemma_1} in the case $n=3$. Given $P(T,Y) \in S'$, it suf\-fices to show that (b) holds for the regular $\Qq$-$S_3$-cover $f : X \rightarrow \mathbb{P}^1$ defined by $P(T,Y)$. Indeed, one then has $S' \subseteq S$ and the desired conclusion then follows from Lemma \ref{lemma_1}. Let $\mathcal{S}_0$ be the set of prime numbers provided by Lemma \ref{lemma_2}. Given $x \geq 1$, denote by $\mathcal{S}'(S_3,x)$ the set of all extensions $F/\Qq$ in $\mathcal{S}(S_3,x)$ which ramify only at prime numbers not in $\mathcal{S}_0$. The asymptotic behaviour of the ratio
$$\frac{|\mathcal{S'}(S_3,x)|}{|\mathcal{S}(S_3,x)|}$$
depends on the Bhargava principle (see \cite{Bha07}), which has been established for $S_3$-extensions of $\Qq$ in \cite{BW08}. 
A consequence of the mass formulae featuring in the principle is that, given a prime number $p$, the set of $S_3$-extensions ramifying tamely at $p$ is (either empty or)\footnote{This first case clearly does not happen, as every prime number ramifies tamely in a suitable $S_3$-extension.} of density at least $c/p$, for some positive constant $c$ not depending on $p$. Furthermore, the principle implies that the probabilities of local behaviours of $S_3$-extensions at any given finite set of prime numbers are independent. This yields
$$\frac{|\mathcal{S'}(S_3,x)|}{|\mathcal{S}(S_3,x)|} = O \bigg(\prod_{\substack{p \leq x\\ {p \in \mathcal{S}_0}}} \bigg( 1 - \frac{1}{p} \bigg) \bigg), \quad x \rightarrow \infty.$$
Then, by Lemma \ref{lemma_2} and \cite[th\'eor\`eme 2.3]{Ser76}, there exists some constant $\alpha >0$ such that
$$\frac{|{\rm{Sp}}(E) \cap \mathcal{{S}}(S_3,x)|}{|\mathcal{{S}}(S_3,x)|} \leq \frac{|\mathcal{S'}(S_3,x)|}{|\mathcal{S}(S_3,x)|} = O({\log}^{- \alpha}(x)), \quad x \rightarrow \infty,$$
where $E/\Qq(T)$ denotes the regular $S_3$-extension associated with the cover $f$. Since the map from ${\rm{Sp}}(f)$ to ${\rm{Sp}}(E)$, mapping a morphism to the fixed field of its kernel, has finite fibers of bounded cardinality (with the bound depending only on the order of the underlying Galois group, which is $6$ here), conclude that (b) holds.
\end{proof}

\begin{remark} \label{rem:general_spaces}
The above way of counting covers is not canonical, since the map between polynomials and covers is not $1$-to-$1$. It does however allow natural generalizations. In particular, assume a family of regular $\Qq$-$G$-covers $X \rightarrow \mathbb{P}^1$ is parameterized by an irreducible polynomial $P(T_1, \dots ,T_k,T,Y)$ with algebraically independent indeterminates $T_1, \dots, T_k$. Such a situation occurs whenever the Hurwitz space of covers of a given ramification type happens to be a rational variety. If, in addition, the branch points of such covers can be chosen such that some element of G$_{\mathbb{Q}}$ permutes them {\textit{without fixed point}}, then our arguments apply in the same way. This idea was used in \cite{Koe17a} to show that most rational translates of a fixed regular $\Qq$-$G$ cover of $\mathbb{P}^1$ have a smaller specialization set than the original cover.
\end{remark}

\section{On a local-global principle for specializations} \label{sec:hasse}

This section deals with our local-global principle for specializations, as alluded to in \S\ref{ssec:intro_4}.

\subsection{Statement of the main result}

We first need some terminology and notation.

Given a prime $\mathfrak{p}$ (possibly infinite) of a number field $k$, denote the restriction map ${\rm{G}}_{k_\mathfrak{p}} \rightarrow {\rm{G}}_k$ by ${\rm{res}}_\mathfrak{p}$ (with $k_\mathfrak{p}$ the completion of $k$ at $\mathfrak{p}$). Given a finite group $G$ and an epimorphism $\varphi : {\rm{G}}_k\rightarrow G$, the composed map $\varphi \circ {\rm{res}}_\mathfrak{p} : {\rm{G}}_{k_\mathfrak{p}} \rightarrow G$ is denoted by $\varphi_\mathfrak{p}$.

\begin{definition} \label{def:loc}
Let $f : X \rightarrow \mathbb{P}^1$ be a regular $\Qq$-$G$-cover and $\varphi : {\rm{G}}_{\Qq} \rightarrow G$ an epimorphism.

\vspace{0.5mm}

\noindent
{\rm{(a)}} Say that $\varphi$ is a specialization morphism of $f$ {\it{everywhere locally}} if $\varphi_p$ is a specialization morphism of $f \otimes_\Qq \Qq_p$ for every prime $p$.

\vspace{0.5mm}

\noindent
{\rm{(b)}} Say that $(f, \varphi)$ {\it{fulfills the local-global principle}} if the following implication holds:
$$\varphi \in {\rm{Sp}}(f)^{\rm{loc}} \Longrightarrow \varphi \in {\rm{Sp}}(f),$$
where ${\rm{Sp}}(f)^{\rm{loc}}$ denotes the set of all epimorphisms ${\rm{G}}_{\Qq} \rightarrow G$ as in (a).
\end{definition}

The existence of an epimorphism $\varphi : {\rm{G}}_{\Qq} \rightarrow G$ such that $(f, \varphi)$ does not fulfill the local-global principle means that $\varphi$ does not occur as a specialization morphism of $f$ but this cannot be detected by local considerations. Moreover, note that a similar principle for specializations of regular $G$-extensions of $\Qq(T)$ could be defined.

This theorem is our main contribution to our local-global principle for specializations:

\begin{theorem} \label{thm:hasse_ab}
Let $f:X\to \mathbb{P}^1$ be a regular $\Qq$-$G$-cover with branch points $t_1, \dots, t_r$. Assume the inertia group at some $t_i$ intersects the center of $G$ non-trivially. Let $q$ be the least prime number such that a central element of order $q$ lies in the inertia group at some $t_i$, and let
\begin{equation} \label{eq:beta}
\beta=\frac{q}{(q-1) |G|}.
\end{equation}
Then the following three conclusions hold.

\vspace{0.5mm}

\noindent
{\rm{(a)}} This inequality holds for some positive constant $C(f)$ and every sufficiently large $x$:
$$|{\rm{Sp}}(f)^{\rm{loc}} \cap \overline{\mathcal{S}}(G,x)| \geq C(f) \cdot {x^{\beta} \cdot \log^{-1}(x)}.$$

\vspace{0.5mm}

\noindent
{\rm{(b)}} Assume the abc-conjecture holds and $r \geq 8$.
Then one has
$$\lim_{x\to \infty} \frac{|{\rm{Sp}}(f) \cap \overline{\mathcal{S}}(G,x)|}{|{\rm{Sp}}(f)^{\rm{loc}} \cap \overline{\mathcal{S}}(G,x)|} = 0.$$
In particular, for some positive constant $C'(f)$ and every sufficiently large integer $x$, the number of epimorphisms $\varphi \in \overline{\mathcal{S}}(G,x)$ such that $(f, \varphi)$ does not fulfill the local-global principle is at least $C'(f) \cdot x^{\beta} \cdot \log^{-1}(x)$.

\vspace{0.5mm}

\noindent
{\rm{(c)}} Assume the abc-conjecture and \eqref{eq2} hold (with $S$ equal to the set of all branch points of $f$), and that the inertia group at some $t_i$ contains a central element of order equal to the least prime divisor of $|G|$. Then one has $\beta=\alpha(G)$, where $\alpha(G)$ is defined in \eqref{eq:intro_3}\footnote{In the general case, one has $\alpha(G) \geq \beta > 1/|G| \geq \alpha(G)/2$.}, and 
$$\lim_{x\to \infty} \frac{|{\rm{Sp}}(f) \cap \overline{\mathcal{S}}(G,x)|}{|{\rm{Sp}}(f)^{\rm{loc}} \cap \overline{\mathcal{S}}(G,x)|} = 0.$$
In particular, for some positive constant $C'(f)$ and every sufficiently large integer $x$, the number of epimorphisms $\varphi \in \overline{\mathcal{S}}(G,x)$ such that $(f, \varphi)$ does not fulfill the local-global principle is at least $C'(f) \cdot x^{\alpha(G)} \cdot \log^{-1}(x)$.
\end{theorem}

\subsection{Proof of Theorem \ref{thm:hasse_ab}} \label{ssec:proof_hasse}

We break the proof into four parts.

\subsubsection{Preliminaries} 

The proof is based on investigation of the local behaviour of specializations of the regular $G$-extension of $\Qq(T)$ associated with $f$. We shall make use of the following general result, stemming from the two papers \cite{DG12} and \cite{KLN19}:

\begin{proposition} \label{prop:DGKLN}
Let $k$ be a number field, $G$ a finite group, $g : X \rightarrow \mathbb{P}^1$ a regular $k$-$G$-cover, $E/k(T)$ the regular $G$-extension associated with $g$, and $t_1, \dots,t_r$ the branch points of $g$. For $1 \leq i \leq r$, let $(E(t_i))_{t_i}/k(t_i)$ be the residue extension of $E(t_i)/k(t_i)(T)$ at the prime ideal generated by $T - t_i$. Then there exists a finite set $\mathcal{S}_{\rm{exc}}$ of prime ideals of the ring of integers of $k$, containing those prime ideals dividing $|G|$, such that, for every prime ideal $\mathfrak{p}$ not contained in $\mathcal{S}_{\rm{exc}}$ and every epimorphism $\varphi: {\rm{G}}_{k} \to G$, the following conclusions hold.

\vspace{0.5mm}

\noindent
{\rm{(a)}} If $\varphi_\mathfrak{p}$ is unramified, then $g \otimes_k k_\mathfrak{p}$ specializes to $\varphi_\mathfrak{p}$.

\vspace{0.5mm}

\noindent
{\rm{(b)}} If $\varphi_\mathfrak{p}$ is totally ramified with image equal to the inertia group at some $t_i$ and if $\mathfrak{p}$ splits completely in the extension $(E(t_i))_{t_i}/k$, then $g \otimes_k k_\mathfrak{p}$ specializes to some homomorphism $\varphi'(\mathfrak{p}) : {\rm{G}}_{k_\mathfrak{p}} \rightarrow G$ such that $\varphi_\mathfrak{p}$ and $\varphi'(\mathfrak{p})$ have the same kernels.
\end{proposition} 

\begin{proof}
(a) follows directly from \cite[Theorem 1.2]{DG12}. As for (b), it is a special case of \cite[Theorem 4.4]{KLN19} (namely, with the assumption $N^{(\mathfrak{p})} = k_\mathfrak{p}$ in the notation there). Note that specialization is worded in terms of fields rather than morphisms in \cite{KLN19}, hence the above conclusion replacing $\varphi_\mathfrak{p}$ by some other morphism with the same kernel.
\end{proof}

We need some notation. Denote the regular $G$-extension of $\Qq(T)$ as\-sociated with $f$ by $E/\Qq(T)$. For $1 \leq i \leq r$, let $(E(t_i))_{t_i}/\Qq(t_i)$ be the residue extension of $E(t_i)/\Qq(t_i)(T)$ at the prime ideal generated by $T - t_i$. 

Let $t_0 \in \mathbb{P}^1(\Qq) \setminus \{t_1, \dots, t_r\}$ be such that the specialization morphism $f_{t_0} : {\rm{G}}_\Qq \rightarrow G$ is surjective; such a $t_0$ exists by Hilbert's irreducibility theorem. The general idea of the proof is to construct, by slightly changing the epimorphism $f_{t_0}$, sufficiently many epimorphisms $\varphi: {\rm{G}}_{\qq}\to G$ that occur as a specialization morphism of $f$ everywhere locally. More precisely, our epimorphisms $\varphi$ will have only one more ramified prime number, compared to $f_{t_0}$. In order to reach the required amount of epimorphisms of bounded discriminant, the newly ramified prime number furthermore needs to have ``small" ramification index. Let $i \in \{1, \dots, r\}$ and $g$ an element of the center of $G$ of order $q$, where $q$ is defined in Theorem \ref{thm:hasse_ab}, such that $g$ is contained in the inertia group at $t_i$.

\subsubsection{Construction of suitable epimorphisms $\varphi : {\rm{G}}_\Qq \rightarrow G$} \label{sssec:construction}

Let $\mathcal{S}_{\rm{exc}}$ be the finite set of prime numbers provided by Proposition \ref{prop:DGKLN}, when applied to the $\Qq$-$G$-cover $f$, $\mathcal{S}$ an arbitrary finite set of prime numbers containing $\mathcal{S}_{\rm{exc}}$, $\mathcal{S}_1$ the set of all prime numbers which ramify in $E_{t_0}/\Qq$, and $p$ a prime number satisfying the following three properties (which depend on $\mathcal{S}$):

\noindent
(i) $p\notin \mathcal{S} \cup \mathcal{S}_1$,

\noindent
(ii) $p$ splits completely in the extension $(E(t_{i}))_{t_{i}} /\Qq$,

\noindent
(iii) $p$ splits completely in $F(\{\sqrt[q]{\ell}  | \ell \in \mathcal{S} \cup \mathcal{S}_1\})/\mathbb{Q}$, where $F=\begin{cases}E_{t_0}(e^{2 i \pi/q}) \text{ if } q\ge 3 \\ E_{t_0}(i) \text{ if } q=2\end{cases}$.

\noindent
In particular, one has $p\equiv 1$ mod $q$ due to (iii).

Let $\varphi(p) :  {\rm{G}}_{\qq}\to \langle g\rangle$ be an epimorphism such that if $L_{(p)}$ denotes the fixed field of the kernel of $\varphi(p)$ in $\overline{\Qq}$, then $L_{(p)}$ is the unique degree $q$ subfield of $\qq(e^{2 i \pi /p})$. Note that the field $L_{(p)}$ embeds into $\Rr$ and the extension $L_{(p)}/\Qq$ ramifies only at $p$ \footnote{In the case $q=2$, we use the fact that $p$ splits in $\mathbb{Q}(i)/\mathbb{Q}$ to ensure that $2$ is unramified in $L_{(p)}/\qq$.}.

Since the ramification loci of $E_{t_0}/\Qq$ and $L_{(p)}/\Qq$ are disjoint (by (i)), the fields $E_{t_0}$ and $L_{(p)}$ are linearly disjoint over $\Qq$. We can therefore consider the direct product homomorphism $\psi(p) = f_{t_0} \times \varphi(p)$; this is an epimorphism from ${\rm{G}}_\qq$ onto $G\times \langle g\rangle$. Let $\Delta$ be the diagonal subgroup of $G\times \langle g\rangle$ generated by $(g,g)$. Note that $\Delta$ is normal as $g$ lies in the center of $G$. Consider the composed map ${\rm{pr}} \circ \psi(p)$, with ${\rm{pr}}$ the canonical projection from $G\times \langle g\rangle$ onto $(G\times \langle g\rangle) / \Delta$. As this quotient group equals $G$ up to canonical isomorphism $g'\mapsto (g',1)\cdot \Delta$, one obtains an epimorphism $\varphi'(p) : {\rm{G}}_\Qq \rightarrow G$.

This  lemma asserts that, up to choosing a suitable set $\mathcal{S}$ and a suitable epimorphism $\varphi(p)$, the above epimorphism $\varphi'(p)$ occurs as a specialization morphism of $f$ everywhere locally:

\begin{lemma} \label{lemma:key}
For some finite set $\mathcal{S} \supseteq \mathcal{S}_{\rm{exc}}$ of prime numbers, depending only on $f$, the following holds. Let $p$ be a prime number satisfying {\rm{(i)}}, {\rm{(ii)}}, and {\rm{(iii)}}. Then there exists an epimorphism $\varphi(p) :  {\rm{G}}_{\qq}\to \langle g\rangle$ with fixed field $L_{(p)}$ as above, and for which the associated epimorphism $\varphi'(p) : {\rm{G}}_\Qq \rightarrow G$ is such that $f \otimes_\Qq \Qq_\ell$ specializes to $\varphi'(p)_\ell$ for every prime $\ell$.
\end{lemma}

\subsubsection{Proof of Theorem \ref{thm:hasse_ab} under Lemma \ref{lemma:key}}

Let $\mathcal{S}$ be a finite set of prime numbers as given by Lemma \ref{lemma:key}. To prove (a), we estimate the number of epimorphisms $\varphi'(p) : {\rm{G}}_\Qq \rightarrow G$ provided by Lemma \ref{lemma:key}, when $p$ runs through the set of all prime numbers satisfying (i), (ii), and (iii). Let $p$ be such a prime number. Since $E_{t_0}$ and $L_{(p)}$ are linearly disjoint over $\Qq$ and as the discriminants $d_{E_{t_0}}$ and $d_{L_{(p)}}$ of $E_{t_0}$ and $L_{(p)}$, respectively, are coprime, one has 
$$|d_{E_{t_0} L_{(p)}}| = |d_{E_{t_0}}|^q \cdot |d_{L_{(p)}}|^{|G|}.$$
Moreover, as $L_{(p)}/\Qq$ is Galois of degree $q$ and ramifies only at $p$, one has $|d_{L_{(p)}}| = p^{q-1}.$ Combine this equality and the fact that $g$ has order $q$ to get that if $L'_{(p)}$ denotes the fixed field of ${\rm{ker}}(\varphi'(p))$ in $\overline{\Qq}$, then $$|d_{L'_{(p)}}| \leq C_1 \cdot p^{\frac{q-1}{q}|G|},$$ 
where $C_1 = |d_{E_{t_0}}|$ depends only on $f$. Furthermore, as $L'_{(p)}/\Qq$ ramifies at $p$ (with ramification index $q$) and is unramified outside $\mathcal{S}_1 \cup \{p\}$, one has 
$L'_{(p_1)} \not= L'_{(p_2)}$
for distinct prime numbers $p_1$ and $p_2$ as above. Finally, as the set of all prime numbers $p$ fulfilling (i), (ii), and (iii) is a positive density subset of the set of all prime numbers, there are asymptotically at least 
$C_2\cdot x \cdot \log^{-1}(x)$
such epimorphisms $\varphi'(p)$ with $p \leq x$, for some positive constant $C_2$ depending only on $f$. In total, the number of such epimorphisms $\varphi'(p)$ with $|d_{L'_{(p)}}| \leq x$ is then asymptotically at least 
$$C_3 \cdot x^{\beta} \cdot {\rm{log}}^{-1}(x),$$
where $C_3$ is a positive constant depending only on $f$. This completes the proof of (a).

As for (b), suppose the abc-conjecture holds and $r \geq 8$. From the latter assumption and the definition of $\beta$, the exponent $e$ defined in \eqref{eq1.5} (with $S$ equal to the set of all branch points of $f$) satisfies $e < \beta(G)$. Pick $\epsilon >0$ with $e + \epsilon < \beta(G)$. Then combine (a) and Theorem \ref{thm:abc_spec} to obtain that
$$\frac{|{\rm{Sp}}(f) \cap \overline{\mathcal{S}}(G,x)|}{|{\rm{Sp}}(f)^{\rm{loc}} \cap \overline{\mathcal{S}}(G,x)|} = O({\log(x)} \cdot x^{e+ \epsilon - \beta(G)}) = o(1), \quad x \rightarrow \infty.$$

Finally, under the assumptions in (c), $q$ is the smallest prime divisor of $|G|$ and one has $\beta=\alpha(G)$. Moreover, in this case, it suffices to check $e < \alpha(G)$ in the proof of (b) above to get the desired conclusion. As seen in the proof of Corollary \ref{coro:abc+malle}, this inequality holds if and only if \eqref{eq2} holds.

\subsubsection{Proof of Lemma \ref{lemma:key}} 

We first prove the following statement:

\begin{lemma} \label{lemma:prelim}
Fix a finite set $\mathcal{S}$ of prime numbers containing $\mathcal{S}_{\rm{exc}}$, a prime number $p$ satisfying {\rm{(i)}}, {\rm{(ii)}}, and {\rm{(iii)}}, and an epimorphism $\varphi(p) :  {\rm{G}}_{\qq}\to \langle g\rangle$ with fixed field $L_{(p)}$ as in \S\ref{sssec:construction}. Then the associated epimorphism $\varphi'(p) : {\rm{G}}_\Qq \rightarrow G$ is such that $\varphi'(p)_\ell$ is a specialization morphism of $f \otimes_\Qq \Qq_\ell$ for every prime $\ell \ne p$.
\end{lemma}

\begin{proof}
Let $\ell$ be a prime $\not=p$. Firstly, assume $\ell$ is finite and $\ell \not \in \mathcal{S} \cup \mathcal{S}_1$. By our construction, $\ell$ is unramified in $E_{t_0} L_{(p)}/ \Qq$, and the same holds in the subextension $L'_{(p)}/\Qq$. Therefore, by Proposition \ref{prop:DGKLN}(a) (which can be applied as $\mathcal{S}$ contains $\mathcal{S}_{\rm{exc}}$), $f \otimes_\Qq \Qq_\ell$ specializes to $\varphi'(p)_\ell$.

Secondly, assume $\ell$ is infinite or $\ell \in \mathcal{S} \cup \mathcal{S}_1$. Then $\varphi(p)_\ell$ is trivial. Indeed, for $\ell = \infty$, this is clear since $L_{(p)} \subseteq \mathbb{R}$ by construction. Assume then that $\ell$ is finite. If $q=2$, then it follows from (iii) and the quadratic reciprocity that $\ell$ is totally split in the extension $\Qq(\sqrt{p})/\Qq = L_{(p)}/\Qq$. One may then assume that $q$ is odd. By (iii), $Y^{q} - \ell$ splits completely in $\qq_p$. This means that $\ell$ is a $q$th power in $\qq_p$. In other words, the multiplicative order of $\ell$ in $\mathbb{F}_p$ is a divisor of $(p-1)/q$. Consequently, the Frobenius of $\Qq(e^{2 i \pi /p})/\Qq$ at $\ell$ is of order dividing $(p-1)/{q}$. As the elements of ${\textrm{Gal}}(\qq(e^{2 i \pi /p})/\qq)$ of order dividing $({p-1})/{q}$ act trivially on $L_{(p)}$, we get that the Frobenius of $L_{(p)}/\qq$ at $\ell$ is trivial, thus proving the claim. Therefore, one has $(f_{t_0})_\ell = \psi(p)_\ell= \varphi'(p)_\ell$. In particular, $f \otimes_\Qq \Qq_\ell$ specializes to $\varphi'(p)_\ell$.
\end{proof}

We now proceed to the proof of Lemma \ref{lemma:key}. By Lemma \ref{lemma:prelim}, it suffices to show that $f \otimes_\Qq \Qq_p$ specializes to $\varphi'(p)_p$, under a suitable choice of $\mathcal{S}$ and $\varphi(p)$. This is done by reducing to Proposition \ref{prop:DGKLN}(b). At this stage, choose $\mathcal{S}$ and $\varphi(p)$ arbitrary as above. 

By the definition of $\varphi(p)$ and (i), the induced epimorphism $\varphi(p)_p: {\rm{G}}_{\qq_p}\to \langle g\rangle$ is totally (tamely) ramified. Its image $\langle g\rangle$ is not necessarily the inertia group at some branch point of $f$. However, this holds for a suitable pullback of $f$. 

Indeed, up to applying a change of variable at the beginning of \S\ref{ssec:proof_hasse}, we may assume $\infty \not \in \{t_1, \dots, t_r\}$. With $U=1/(T-t_{i})$, one sees that $\infty$ is a branch point of $E(t_{i})/\Qq(t_{i})(U)$ but 0 is not. Let $e_i$ be the ramification index at $t_i$ and $V=U^{q/e_i}$. Since the extensions $E \overline{\qq}/\overline{\qq}(U)$ and $\overline{\qq}(V)/\overline{\qq}(U)$ have only one common branch point (namely, $\infty$), the fields $E \overline{\qq}$ and $\overline{\qq}(V)$ are linearly disjoint over $\overline{\qq}(U)$. Thus, $E \overline{\qq}(V)/\overline{\qq}(V)$ is still a regular $G$-extension, and the same holds, in particular, for  $E(t_i)(V)/\qq(t_i)(V)$. Let $f' : X' \rightarrow \mathbb{P}^1$ be the associated regular $\Qq(t_{i})$-$G$-cover. By Abhyankar's lemma, $\langle g \rangle$ is the inertia group of $f'$ at $\infty$.

Set $k= (E(t_{i}))_{t_{i}}$ and denote the cover $f' \otimes_{\Qq(t_{i})} k$ by $f''$. If $\mathfrak{p}$ is any prime ideal lying over $p$ in $k/\Qq$, then the completion $k_\mathfrak{p}$ is equal to $\Qq_p$, due to the splitting assumption in (ii). Moreover, as $p$ is totally split in $E_{t_0}/\Qq$ (by (iii)), the restriction $({f_{t_0}})_p$ is trivial, that is, $\varphi(p)_p = \psi(p)_p = \varphi'(p)_p$. Hence, since every specialization morphism of $f'' \otimes_k \Qq_p$ (at a point $t \in \mathbb{P}^1(\Qq_p))$ is a specialization morphism of $f \otimes_\Qq \Qq_p$ (namely, at $1/t^{e_i/q} + t_{i}$), it suffices to show that $f'' \otimes_k \Qq_p$ specializes to $\varphi(p)_p$. 

Choose $\mathcal{S}$ as the set of all prime numbers $\ell$ which are in the already defined set $\mathcal{S}_{\rm{exc}}$ or such that some prime ideal lying over $\ell$ in the extension $k/\Qq$ belongs to the exceptional set provided by Proposition \ref{prop:DGKLN}, when applied to the cover $f''$. By Proposition \ref{prop:DGKLN}(b), $f'' \otimes_k \Qq_p$ specializes to some homomorphism $\varphi''(p) : {\rm{G}}_{\mathbb{Q}_p} \rightarrow G$ such that $\varphi(p)_p$ and $\varphi''(p)$ have the same kernels. In particular, the image of $\varphi''(p)$ is equal to $\langle g \rangle$ and $\varphi''(p) = \sigma \circ \varphi(p)_p$ for some automorphism $\sigma : \langle g \rangle \rightarrow \langle g \rangle$. Then consider the epimorphism $\sigma \circ \varphi(p) : {\rm{G}}_\Qq \rightarrow \langle g \rangle$. The fixed field of the kernel of this epimorphism is equal to that of $\varphi(p)$ and $f'' \otimes_k \Qq_p$ specializes to $(\sigma \circ \varphi(p))_p$, as $(\sigma \circ \varphi(p))_p = \sigma \circ \varphi(p)_p = \varphi''(p)$. Conclude that the lemma holds.

\subsection{On the assumptions of Theorem \ref{thm:hasse_ab}} \label{ssec:hasse_covers}

We exhibit below several explicit situations where covers as in Theorem \ref{thm:hasse_ab} can be constructed.

\subsubsection{Abelian groups} \label{sssec:abelian}

If $G$ is an arbitrary finite abelian group, then the condition on inertia groups in Theorem \ref{thm:hasse_ab}(c) is satisfied for every regular $\Qq$-$G$ cover $f$ of $\mathbb{P}^1$, as the inertia groups at the branch points of $f$ generate $G$. Morever, if $f$ has at least 7 branch points, then \eqref{eq2} holds (see Remark \ref{rk:eq3}). Hence, the conclusion of Theorem \ref{thm:intro_6} follows from Theorem \ref{thm:hasse_ab}(c).

\subsubsection{Extension to some non-abelian groups} \label{sssec:non-ab}

In fact, the same applies for some non-abelian groups $G$ as well. Here are some examples:

\noindent
(a) $G =H \times \zz/2^k\zz$, where $H$ is an arbitrary finite group, and $2^k$ is strictly larger than the highest $2$-power occurring as an element order in $H$,

\noindent
(b) $G=Q_8^{n} \times H$, where $Q_8$ is the quaternion group, $n \geq 1$, and $H$ is abelian.

Indeed, for (a), if $(h_1, g_1), \dots, (h_r, g_r)$ generate $G$, with $h_1, \dots, h_r \in H$ and $g_1, \dots, g_r \in \Zz/2^k \Zz$, then we may assume $g_1$ is of order $2^k$. Thanks to our assumption on $k$, there exists $m \geq 1$ with $h_1^{2^{k-1} + m 2^{k}}=1$. As for (b), suppose $(g_1, h_1), \dots, (g_r, h_r)$ generate $G$, with $g_1, \dots, g_r \in Q_8^n$ and $h_1, \dots, h_r \in H$. We may assume $g_1$ is of order 4. Then $(g_1^2, h_1^2)$ has even order, say $2m$ with $m \geq 1$. Hence, $(g_1^{2m}, h_1^{2m})$ has order 2 and is in the center of $G$.

\subsubsection{Regular Galois groups over $\Qq$ with non-trivial center} \label{sssec:hasse_regular}

Let $G$ be a regular Galois group over $\Qq$ with non-trivial center. Then there exists a regular $\Qq$-$G$-cover of $\mathbb{P}^1$ whose inertia group at some branch point intersects the center of $G$ non-trivially.

Indeed, let $f : X \rightarrow \mathbb{P}^1$ be a regular $\Qq$-$G$-cover, $g$ an element of the center of $G$ of prime order, and $f': X' \rightarrow \mathbb{P}^1$ a regular $\Qq$-$\langle g \rangle$-cover. Up to applying a suitable change of variable, we may assume that the sets of branch points of $f$ and $f'$ are disjoint. Denote the function field extensions of the covers $f$ and $f'$ by $E/\Qq(T)$ and $E'/\Qq(T)$, respectively. Then the fields $E \overline{\Qq}$ and $E' \overline{\Qq}$ are linearly disjoint over $\overline{\Qq}$, that is, the extension $EE'/\Qq(T)$ is a regular $(G \times \langle g \rangle)$-extension. If $E"$ denotes the fixed field of $\langle g \rangle \times \langle g \rangle$ in $EE'$,
then $E"/\Qq(T)$ is a regular $G$-extension, each branch point $t$ of $E'/\Qq(T)$ is a branch point of $E"/\Qq(T)$, and the inertia group of $E"/\Qq(T)$ at $t$ is equal to $\langle g \rangle$. Consequently, the regular $\Qq$-$G$-cover $f"$ of $\mathbb{P}^1$ associated with the extension $E"/\Qq(T)$ satisfies the desired conclusion.

We note for later use that we can simultaneously require that no branch point of $f"$ is $\Qq$-rational and the total number of branch points of $f"$ is arbitrarily large (in particular, at least 8). Indeed, up to replacing $f$ by a suitable pullback of $f$, we may assume no branch point of $f$ is $\Qq$-rational. Moreover, as a consequence of the rigidity method, we may assume the same holds for the cover $f'$ and the number of branch points of $f'$ is arbitrarily large.

\section{Diophantine aspects} \label{sec:diophantine}

In this section, we discuss diophantine aspects of our results, as already alluded to in \S\ref{ssec:intro_5}.

\subsection{Preliminaries} \label{ssec:diophantine_prelim}

Our first aim is to briefly recall the definition and the main properties of the twisted cover from \cite{Deb99a}. See, e.g., \cite[\S2.2]{DG12} for more details. We use below the notation introduced in \S\ref{ssec:basics_1}.

Let $k$ be a field of characteristic zero, $\overline{k}$ an algebraic closure of $k$, $G$ a finite group, $f : X \rightarrow \mathbb{P}^1$ a regular $k$-$G$-cover of branch points $t_1, \dots, t_r$, and $\varphi :{\rm{G}}_k \rightarrow G$ a homomorphism. Denote the right-regular (resp., left-regular) representation of $G$ by $\delta: G \rightarrow S_{|G|}$ (resp., by $\gamma: G \rightarrow S_{|G|}$). Define $\varphi^* : {\rm{G}}_k \rightarrow G$ by $\varphi^*(\sigma) = \varphi(\sigma)^{-1}$ ($\sigma \in {\rm{G}}_k$). Denote the restriction map $\pi_1(\mathbb{P}^1 \setminus \{t_1, \dots,t_r\},t)_k \rightarrow {\rm{G}}_k$ by ${\rm{res}}$
and the multiplication in $S_{|G|}$ by $\times$.

If $\phi : \pi_1(\mathbb{P}^1 \setminus \{t_1, \dots,t_r\},t)_k \rightarrow G$ is the epimorphism corresponding to $f$, consider 
$$\widetilde{\phi}^\varphi : \left \{ \begin{array} {ccc}
\pi_1(\mathbb{P}^1 \setminus \{t_1, \dots,t_r\},t)_k & \longrightarrow & S_{|G|} \\
\theta & \longmapsto & \widetilde{\phi}^\varphi(\theta) = \gamma \circ \phi (\theta) \times \delta \circ \varphi^* \circ {\rm{res}} (\theta).
\end{array} \right. $$
Then the map $\widetilde{\phi}^\varphi$ is a homomorphism with the same restriction to $\pi_1(\mathbb{P}^1 \setminus \{t_1, \dots,t_r\},t)_{\overline{k}}$ as $\phi$, hence corresponds to a regular $k$-cover (not Galois in general), denoted by $\widetilde{f}^{\varphi} : \widetilde{X}^\varphi \rightarrow \mathbb{P}^1$ and called the {\it{twisted cover}} of $f$ by $\varphi$, which satisfies $f \otimes_k \overline{k} = \widetilde{f}^\varphi \otimes_k \overline{k}$. In particular, the covers $f$ and $\widetilde{f}^\varphi$ have the same branch points. 

The following proposition (see \cite[Twisting Lemma 2.1]{DG12}) contains the main property of the twisted cover:

\begin{proposition} \label{tl}
For every $t_0 \in \mathbb{P}^1(k) \setminus \{t_1, \dots, t_r\}$, the following conditions are equivalent:

\noindent
{\rm{(a)}} there exists a $k$-rational point $x_0$ on $\widetilde{X}^\varphi$ such that $\widetilde{f}^\varphi(x_0) = t_0$,

\vspace{0.5mm}

\noindent
{\rm{(b)}} there exists $\omega \in G$ such that the specialization morphism $f_{t_0}$ equals ${\rm{conj}}(\omega) \circ \varphi$.
\end{proposition}

Furthermore, the twisting operation commutes with extension of scalars: if $k' \supseteq k$, then the twisted cover of $f \otimes_k k'$ by the restriction of $\varphi$ to ${\rm{G}}_{k'}$ \footnote{i.e., by the homomorphism $\varphi \circ {\rm{res}}' : {\rm{G}}_{k'} \rightarrow G$, where ${\rm{res}}' : {\rm{G}}_{k'} \rightarrow {\rm{G}}_k$ is the restriction map.}
is the regular $k'$-cover $\widetilde{f}^\varphi \otimes_k k'$.

Condition (a) of Proposition \ref{tl} leads us to the following terminology:

\begin{definition} \label{def:trivial}
Let $f : X \rightarrow \mathbb{P}^1$ be a regular $k$-cover. Say that a $k$-rational point $x$ on $X$ is {\it{trivial}} if $f(x)$ is a ($k$-rational) branch point of $f$, and {\it{non-trivial}} otherwise.
\end{definition}

\begin{example} \label{ex:quadra}
Let $f : X \rightarrow \mathbb{P}^1$ be a regular $\Qq$-$\Zz/2\Zz$-cover and $P(T) \in \Zz[T]$ separable such that $X$ is the hyperelliptic curve $C_{P(T)} : y^2=P(t)$. Then the set of all epimorphisms $\varphi : {\rm{G}}_{\Qq} \rightarrow \Zz/2\Zz$ is in 1-to-1 correspondence with the set $\mathcal{N}_2$ of all squarefree integers. Given such an integer $d$, the associated twisted curve $\widetilde{X}^{\varphi}$ is the hyperelliptic curve $C_{d \cdot P(T)} : y^2= d \cdot P(t)$. Moreover, trivial points in the sense of Definition \ref{def:trivial} correspond to those defined in \S\ref{ssec:basics_2}, and Proposition \ref{prop:Z/2Z}(b) corresponds to the quadratic case of Proposition \ref{tl}.
\end{example}

\subsection{Global aspects} \label{ssec:diophantine_global}

Let $G$ be a finite group and $f : X \rightarrow \mathbb{P}^1$ a regular $\Qq$-$G$-cover. By \S\ref{ssec:diophantine_prelim}, the set ${\rm{Sp}}(f)$ is the set of all 
homomorphisms $\varphi : {\rm{G}}_\Qq \rightarrow G$ such that the twisted curve $\widetilde{X}^\varphi$ has a non-trivial $\Qq$-rational point. Hence, Theorem \ref{thm:abc_spec} can be rephrased as follows: 

\begin{theorem} \label{thm:abc_spec2}
Let $S \subseteq \mathbb{P}^1(\overline{\qq})$ be a non-empty subset of the set of branch points of $f$, closed under the action of {\rm{G}}$_\Qq$. Assume the abc-conjecture and \eqref{eq1} hold. Then, for every $\epsilon>0$ and every sufficiently large $x$, the number $h(x)$ of all epimorphisms $\varphi : {\rm{G}}_\Qq \rightarrow G$ in $\overline{\mathcal{S}}(G,x)$ such that the twisted curve $\widetilde{X}^\varphi$ has a non-trivial $\Qq$-rational point satisfies 
$$h(x) \leq x^{e + \epsilon},$$
where $e$ is defined in \eqref{eq1.5}.
\end{theorem}

Similarly, all other results from \S\ref{sec:densityI} and \S\ref{sec:densityII} with a density zero conclusion can be rewritten with the above diophantine flavour. We leave this to the interested reader.

In the case $G=\Zz/2\Zz$, Theorem \ref{thm:abc_spec2} yields this corollary, which is \cite[Corollary 1]{Gra07}:

\begin{corollary} \label{coro:granville}
Let $P(T) \in \Zz[T]$ be a separable polynomial of degree $ \geq 5$ and $g$ the genus of the hyperelliptic curve $C_{P(T)} : y^2=P(t)$. Assume the abc-conjecture holds. Then, for every $\epsilon>0$ and every sufficiently large $x$, the number $h(x)$ of all squarefree integers $d \in \llbracket -x , x \rrbracket$ \footnote{Recall that $\llbracket -x , x \rrbracket$ denotes the set of all integers between $-x$ and $x$.} such that the twisted curve $C_{d \cdot P(T)} : y^2 = d \cdot P(t)$ has a non-trivial $\Qq$-rational point satisfies 
$$h(x) \leq x^{(1/(g-1)) + \epsilon}.$$
\end{corollary}

\begin{proof}
Let $f: X \rightarrow \mathbb{P}^1$ be the regular $\Qq$-$\Zz/2\Zz$-cover given by the polynomial $Y^2-P(T)$. By Proposition \ref{prop:Z/2Z}(a), $f$ has $r \geq 6$ branch points. Hence, by Remark \ref{rk:stup}(b), Example \ref{ex:quadra}, and Theorem \ref{thm:abc_spec2}, it suffices to show that the exponent $e$ (with $S$ the set of all branch points of $f$) is equal to $1/(g-1)$. By \eqref{eq1.5}, one has $e=2/(r-4)$. Moreover, one has $2(g-1)=r-4$ by the Riemann-Hurwitz formula. Conclude that the desired equality holds.
\end{proof}

\subsection{On the local-global principle for specializations} \label{ssec:diophantine_local}

For a regular $\Qq$-$G$-cover $f : X \rightarrow \mathbb{P}^1$, \S\ref{ssec:diophantine_prelim} shows that the set ${\rm{Sp}}(f)^{\rm{loc}} \setminus {\rm{Sp}}(f)$, with ${\rm{Sp}}(f)^{\rm{loc}}$ introduced in Definition \ref{def:loc}, is the set of all epimorphisms $\varphi : {\rm{G}}_\Qq \rightarrow G$ such that the twisted curve $\widetilde{X}^\varphi$ has a non-trivial $\Qq_p$-rational point for every prime $p$ but only trivial $\Qq$-rational points. As above, Theorem \ref{thm:hasse_ab} may be worded with this diophantine flavour. We leave this to the interested reader.

Let us rather give an application of our result to the Hasse principle. Recall that a curve $C$ over $\Qq$ fulfills the Hasse principle if the following implication holds:

\vspace{1.5mm}

\centerline{{\it{$C$ has a $\Qq_p$-rational point for every prime $p$}} $\Longrightarrow$ {\it{$C$ has a $\Qq$-rational point.}}}

\vspace{1.5mm}

\noindent
The sole difference between the Hasse principle and the diophantine analog of our local-global principle for specializations is that, since we are interested in covers rather than just abstract curves, we have to disallow rational points extending branch points. However, if we start with a cover with no $\Qq$-rational branch point, then the twisted curves provided by (the diophantine version of) Theorem \ref{thm:hasse_ab} do not fulfill the Hasse principle.

For example, by combining Theorem \ref{thm:hasse_ab} and \S\ref{ssec:hasse_covers}.1-2, we obtain Corollary \ref{thm:hasse_exp} below, which makes Theorem \ref{thm:intro_7} more precise:

\begin{corollary} \label{thm:hasse_exp}
Let $G$ be a finite abelian group or a finite group as in \S\ref{sssec:non-ab}, and let $f:X\to \mathbb{P}^1$ be a regular $\Qq$-$G$ cover with no $\Qq$-rational branch point. Assume the abc-conjecture holds and $f$ has at least 7 branch points. Then, for some positive constant $C(f)$ and  every sufficiently large $x$, the number of epimorphisms $\varphi \in \overline{\mathcal{S}}(G,x)$ such that $\widetilde{X}^\varphi$ does not fulfill the Hasse principle is at least 
$$C(f) \cdot x^{\alpha(G)} \cdot \log^{-1}(x),$$
where $\alpha(G)$ is defined in \eqref{eq:intro_3}.
\end{corollary}

In the special case $G=\Zz/2\Zz$, we have this corollary, which is \cite[Theorem 2]{CW18} and which follows from Corollary \ref{thm:hasse_exp} as Corollary \ref{coro:granville} follows from Theorem \ref{thm:abc_spec2}:

\begin{corollary} \label{coro:CW18}
Let $P(T) \in \Zz[T]$ be a separable polynomial of even degree at least 8 and without any root in $\Qq$. Suppose the abc-conjecture holds. Then there exists a positive constant $C$, depending only on $P(T)$, which satisfies the following. For every sufficiently large $x$, the number of all squarefree integers $d \in \llbracket -x , x \rrbracket$ such that the twisted hyperelliptic curve $C_{d \cdot P(T)} : y^2 = d \cdot P(t)$ does not fulfill the Hasse principle is at least $C \cdot x \cdot \log^{-1}(x)$.
\end{corollary}

\begin{remark}
If $P(T)$ is of odd degree, then the conclusion fails trivially as the trivial point $[0:1:0]$ lies on every quadratic twist of $C_{P(T)}$. This actually gives an example where the Hasse principle holds but our local-global principle fails.

Namely, consider a separable polynomial $P(T) \in \Zz[T]$ of odd degree. Then $C_{P(T)}: y^2=P(t)$ has a non-trivial $\Qq_p$-rational point for every prime $p$ (an easy consequence of Hensel's lemma). Consequently, if $d$ denotes an arbitrary squarefree integer, then the twisted hyperelliptic curve $C_{d \cdot P(T)} : y^2 = d \cdot P(t)$ has a non-trivial $\Qq_p$-rational point for every prime $p$. However, by Corollary \ref{coro:granville}, if $P(T)$ has degree at least 7 and the abc-conjecture holds, then $C_{d \cdot P(T)}$ has only trivial $\Qq$-rational points for almost all squarefree integers $d$. Note that this last conclusion does hold unconditionally for infinitely many squarefree integers $d$ in some situations (see, e.g., the upcoming Proposition \ref{prop:hyper_uncond}).

Even though the above situation seems like an artificial creation of a failure of Hasse principle (by disallowing trivial points), it is important from our point of view of specializations, since it yields a case of a regular $\Qq$-$G$-cover $f : X \rightarrow \mathbb{P}^1$ where {\it{every}} epimorphism $\varphi : {\rm{G}}_\Qq \rightarrow G$ is a specialization morphism of $f$ everywhere locally, but not all such $\varphi$'s occur as specialization morphisms of $f$. In particular, it provides a conditional example where the ratio \eqref{eq:ratio_2} tends to $0$, whereas the ratio in \eqref{eq:ratio} does not.
\end{remark}

In fact, Corollary \ref{thm:hasse_exp} holds if $f$ an arbitrary regular $\Qq$-$G$-cover of $\mathbb{P}^1$ with 8 branch points or more, none of them is $\Qq$-rational, and such that some geometric inertia group contains a non-trivial central element\footnote{at the cost of replacing $\alpha(G)$ by the smaller constant $\beta$ defined in \eqref{eq:beta}.}. In particular, this corollary, which relies on  \S\ref{sssec:hasse_regular}, allows to conditionally generate many more curves over $\Qq$ failing the Hasse principle:

\begin{corollary} \label{coro:hasse_existence}
Let $G$ be a regular Galois group over $\Qq$ with non-trivial center. Assume the abc-conjecture holds. Then there exist a curve $C$ over $\Qq$, with a regular $\Qq$-$G$-cover to $\mathbb{P}^1$, and ``many" $\Qq$-curves $C'$, which are isomorphic to $C$ up to base change from $\Qq$ to $\overline{\Qq}$ and which do not fulfill the Hasse principle.
\end{corollary}

Note that our arguments indeed yield infinitely many pairwise non-isomorphic (over $\mathbb{Q}$) such curves $C'$. This is because isomorphic curves over $\mathbb{Q}$ have isomorphic function fields, whereas it is easy to see that twists $\widetilde{f}^{\varphi_1}$ and $\widetilde{f}^{\varphi_2}$ of the same cover $f$ have non-isomorphic function fields as soon as the kernels of $\varphi_1$ and $\varphi_2$ have distinct fixed fields.

\appendix

\section{Parametric extensions with few branch points} \label{sec:parametric}

The aim of this section is to use various tools from previous papers to prove the following conditional result about parametric extensions with at most three branch points:

\begin{theorem} \label{prop:r_small}
Let $G$ be a non-trivial finite group and let $E/\Qq(T)$ be a regular $G$-extension with $r \leq 3$ branch points.

\vspace{0.25mm}

\noindent
{\rm{(a)}} Suppose $r=2$. Then the following three conditions are equivalent:

\vspace{0.25mm}

{\rm{(1)}} the extension $E/\Qq(T)$ is generic,

\vspace{0.25mm}

{\rm{(2)}} the extension $E/\Qq(T)$ is parametric,

\vspace{0.25mm}

{\rm{(3)}} either $E=\Qq(\sqrt{T})$, up to a applying a change of variable (that is, $G=\Zz/2\Zz$ and each

branch point of $E/\Qq(T)$ is $\Qq$-rational), or $G=\Zz/3\Zz$.

\vspace{0.25mm}

\noindent
{\rm{(b)}} Suppose all finite groups occur as Galois groups over $\Qq$ and $r=3$. Then the following three conditions are equivalent:

\vspace{0.25mm}

{\rm{(1)}} the extension $E/\Qq(T)$ is generic,

\vspace{0.25mm}

{\rm{(2)}} the extension $E/\Qq(T)$ is parametric,

\vspace{0.25mm}

{\rm{(3)}} the field $E$ is equal to the splitting field over $\Qq(T)$ of the polynomial $Y^3+TY+T$ (in

which case $G=S_3$), up to applying a change of variable.
\end{theorem}

\begin{proof}
(a) First, assume $r=2$. By the Riemann Existence Theorem, one has $G= \Zz/n\Zz$ for some $n \geq 2$. As in the proof of Corollary \ref{coro:abc_explicit}, the Branch Cycle Lemma yields $2=r \geq \varphi(n)$, where $\varphi$ denotes the Euler totient function, that is, $n \in \{2,3,4,6\}$. First, assume $n=2$. Then $E/\Qq(T)$ is parametric if and only if each branch point is $\Qq$-rational \cite[Proposition 3.1]{Leg15}, and it is clear that $E/\Qq(T)$ is generic if the latter holds. Now, assume $n=3$. Then, as a consequence of, e.g., \cite[Proposition 5.3]{DKLN18}, the extension $E/\Qq(T)$ is generic. Finally, assume $n \in \{4,6\}$. Then, by the Branch Cycle Lemma, none of the branch points of $E/\Qq(T)$ is $\Qq$-rational. In particular, there exist infinitely many prime numbers which ramify in no specialization of $E/\Qq(T)$; see Lemma \ref{lemma_2}. However, for all prime numbers $p$, there are $\Zz/n\Zz$-extensions of $\Qq$ which ramify at $p$. Conclude that $E/\Qq(T)$ is not parametric.

\vspace{0.5mm}

\noindent
(b) Now, we suppose $r=3$. As all finite groups have been assumed to be Galois groups over $\Qq$, one may use \cite[Proposition 1]{KM01} to get that there exists a totally real $G$-extension of $\Qq$. By \cite[Proposition 1.2]{DF90}, such a $G$-extension of $\Qq$ cannot occur as a specialization of $E/\Qq(T)$, unless $G$ is dihedral of order $2n$ with $n \in \{2,3,4,6\}$. 

First, assume $G$ is dihedral of order $2n$ with $n \in \{2,4,6\}$. In each case, $G$ has a non-cyclic abelian subgroup (namely, $\Zz/2\Zz \times \Zz/2\Zz$). Then recall that, in this situation, \cite[Theorem 6.2]{KLN19} shows that the extension $E/\Qq(T)$ cannot be ``locally parametric". That is, for infinitely many prime numbers $p$, there exists a finite Galois extension $F_p/\Qq_p$ whose Galois group embeds into $G$ and which does not occur as a specialization of $E\Qq_p/\Qq_p(T)$. Since $G$ is dihedral, up to dropping finitely many such primes, such a Galois extension $F_p/\Qq_p$ can be lifted to a $G$-extension $F/\Qq$, that is, the field $F_p$ is the completion of $F$ at $p$; see \cite[Theorem 1.1]{DLN17}. In particular, the extension $F/\Qq$ cannot occur as a specialization of $E/\Qq(T)$. 

Now, assume $G=S_3$. One easily checks that the ramification indices of the branch points of $E/\Qq(T)$ are 2, 2, and 3, i.e., the inertia groups of the branch points are generated by a 2-cycle, a 2-cycle, and a 3-cycle. Let $C_2$ (resp., $C_3$) be the conjugacy class in $S_3$ of the 2-cycles (resp., of the 3-cycles). If the first two branch points are not $\Qq$-rational, then, by (a), the quadratic subextension of $E/\qq(T)$ is not parametric. Since every quadratic number field embeds into an $S_3$-extension of $\qq$, this implies that $E/\qq(T)$ cannot be parametric either. So all three branch points can be assumed to be $\Qq$-rational. Since $(C_2, C_2, C_3)$ is a rigid triple of rational conjugacy classes of the centerless group $S_3$, there is only one regular $S_3$-extension of $\Qq(T)$ with 3 $\Qq$-rational branch points, up to change of variable. See, e.g., \cite[Chapters 7 and 8]{Ser92} for more details. Let $E'$ be the splitting field of $Y^3+TY+T$ over $\Qq(T)$. Since $E'/\Qq(T)$ is a regular $S_3$-extension and its set of branch points is $\{0, \infty, -27/4\}$, it is the only regular $S_3$-extension of $\Qq(T)$ with 3 $\Qq$-rational branch points. As this extension is known to be generic (see, e.g., \cite[\S2.1]{JLY02} or \cite[Proposition 5.3]{DKLN18}), we are done. 
\end{proof}

\begin{remark}
(a) We do not know whether the equivalence between ``$E/\Qq(T)$ parametric" and ``$E/\Qq(T)$ generic" holds without assuming the number of branch points is at most 3 and every finite group occurs as a Galois group over $\Qq$ \footnote{Over larger number fields $k$, examples of regular $G$-extensions of $k(T)$ which are parametric but not generic are known, under the Birch and Swinnerton-Dyer conjecture. See \cite[\S5.4]{DKLN18} for more details.}. Note that this result would imply that only the subgroups of $S_3$ have a parametric extension over $\Qq$, since this last conclusion holds with ``parametric" replaced by ``generic"; see \cite{JLY02} and \cite[Corollary 5.4]{DKLN18}.

\vspace{0.5mm}

\noindent
(b) Given a finite group $G$, every regular $G$-extension of $\Qq(T)$ of genus 0 has at most 3 branch points (by the Riemann-Hurwitz formula). Hence, Theorem \ref{prop:r_small} shows that, under a positive answer to the inverse Galois problem, any given regular $G$-extension $E/\Qq(T)$ of genus 0 which is parametric is generic. This weaker conclusion actually holds unconditionally.

Indeed, denote the number of branch points of $E/\Qq(T)$ by $r$. Since $E/\Qq(T)$ has genus 0, one of these conditions holds:

(1) $G$ is cyclic of order $n \in \{2,3,4,6\}$ and $r=2$,

(2) $G$ is dihedral of order $2n$ with $n \in \{2,3,4,6\}$ and $r=3$,

(3) $G=A_4$ and $r=3$,

(4) $G=S_4$ and $r=3$ \footnote{Indeed, since $E/\Qq(T)$ is of genus 0, the group $G$ embeds into ${\rm{PGL}}_2(\overline{\Qq})$ and, by \cite[Chapter I, Theorem 6.2]{MM18}, we get that one of the following five conditions holds: (1) $G$ is cyclic and $r=2$, (2) $G$ is dihedral and $r=3$, (3) $G=A_4$ and $r=3$, (4) $G=S_4$ and $r=3$, and (5) $G=A_5$ and $r=3$. First, as in the proof of Theorem \ref{prop:r_small}(a), (1) can happen only if $n \in \{2,3,4,6\}$, by the Branch Cycle Lemma. Now, (5) cannot happen. Indeed, the ramification indices of the branch points of $E/\Qq(T)$ should be 2, 3, and 5 (see \cite[Chapter I, Theorem 6.2]{MM18}), thus violating the Branch Cycle Lemma since $A_5$ has two conjugate conjugacy classes of 5-cycles. Finally, in the case of dihedral groups, similar arguments show that (2) can happen only if $n \in \{2,3,4,6\}$.}.

\noindent
If (1) holds, then $E/\Qq(T)$ is generic iff it is parametric, by Theorem \ref{prop:r_small}(a). If (2) (with $n \not=3$) or (3) or (4) holds, then $G$ has a non-cyclic abelian subgroup. One then shows as in the proof of Theorem \ref{prop:r_small}(b) that $E/\Qq(T)$ is not parametric. Finally, if (2) holds with $n=3$, then one sees as above that $E/\Qq(T)$ is non-parametric or $E$ is the splitting field over $\Qq(T)$ of $Y^3+TY+T$, up to change of variable.
\end{remark}

\section{Twists of superelliptic curves without rational points} \label{app:proof}

\subsection{Proof of Theorem \ref{thm:super_uncond}} \label{app:proof_1}

Let $S'$ be the subset of $\mathcal{P}(n,N)$ consisting of all polynomials $P(T)$ satisfying this condition:

\vspace{1mm}

\noindent
{\rm{($*$)}} {\it{$P(T)$ is separable and $\bigcup_{j=1}^N {\rm{Gal}}(L/\Qq(t_j)) \not= {\rm{Gal}}(L/\Qq)$, where $t_1, \dots, t_N$ and $L$ are the roots and the splitting field over $\Qq$ of $P(T)$, respectively.}}

\vspace{1mm}

First, an element $P(T)$ of $\mathcal{P}(n,N)$ is in $S'$ if its Galois group over $\Qq$, viewed as a permutation group of the roots, is isomorphic to $S_N$. One then shows as in the proof of Lemma \ref{lemma_0.5} that the estimate \eqref{eq:bound0} holds. Moreover, if $P(T) \in S'$, then, as in the proof of Lemma \ref{lemma_2}, there is a set $\mathcal{S}$ of prime numbers of positive density $\alpha$ such that no prime number $p \in \mathcal{S}$ is a prime divisor of $P(T)$ \footnote{The definition of a prime divisor of a polynomial is recalled in the proof of Lemma \ref{lemma_2}.}. Set $P(T) = a_0 + a_1 T + \cdots + a_N T^N.$ As condition ($*$) holds, $P(T)$ has no root in $\Qq$. In particular, one has $a_0 \not=0$. Up to dropping finitely many prime numbers, we may assume $v_p(a_0) = 0$ and $v_p(a_N)=0$ for each prime number $p \in \mathcal{S}$. 

Next, let $d$ be an arbitrary $n$-free number which is divisible by at least one prime number $p \in \mathcal{S}$. Suppose $C_{d \cdot P(T)}$ has a (non-trivial) $\Qq$-rational point $[y:t:z]$. If $z=0$, one has
\begin{equation} \label{eq-1}
y^n = d \cdot a_N t^N.
\end{equation}
In particular, one has $y \not=0$ and $t \not=0$. By the condition $v_p(a_N)=0$ and \eqref{eq-1}, one has $$n \cdot v_p(y) = v_p(d) + N \cdot v_p(t).$$
As $n$ divides $N$, we get that $n$ divides $v_p(d)$, which cannot happen. One then has $z \not=0$. Up to replacing $(y,t,z)$ by $(y/z^{N/n}, t/z, 1)$, we may assume $z=1$. Hence, one has
\begin{equation} \label{eq 0}
y^n = d \cdot P(t).
\end{equation}
 If $v_p(P(t)) = 0$, \eqref{eq 0} gives $n \cdot v_p(y) = v_p(d)$. Then $n \vert v_p(d)$, which cannot happen. Hence,
\begin{equation} \label{eq -1-}
v_p(P(t)) \not=0.
\end{equation}
If $t=0$, \eqref{eq 0} gives $y^n = d \cdot a_0$. Since $v_p(a_0) = 0$, we get $n \vert v_p(d)$, a contradiction. Hence, $t \not=0$. If $v_p(t) \geq 0$, then $v_p(P(t)) \geq 0$. By \eqref{eq -1-}, this yields $v_p(P(t)) >0$. Then $p$ is a prime divisor of $P(T)$, a contradiction. Hence, $v_p(t)<0$. Using that $v_p(a_N)=0$, we get 
\begin{equation} \label{eq 2}
v_p(P(t)) = v_p(t^N) = N \cdot v_p(t).
\end{equation}
Combining \eqref{eq 0} and \eqref{eq 2} then provides
$n \cdot v_p(y) = v_p(d) +N \cdot v_p(t).$
As $n | N$, we get that $n$ divides $v_p(d)$, which cannot happen. One then has $C_{d \cdot P(T)}(\Qq) = \emptyset$.

Finally, let $\mathcal{N}_\mathcal{S}$ be the set of all integers $d$ which are divisible by no prime number in $\mathcal{S}$. By the above, one has $|\mathcal{N}_{n}(P(T),x)| \leq |\mathcal{N}_{\mathcal{S}} \cap \llbracket -x,x \rrbracket|$ for every positive integer $x$. Moreover, by \cite[th\'eor\`eme 2.3]{Ser76}, one has
$|\mathcal{N}_{\mathcal{S}} \cap \llbracket -x,x \rrbracket| \sim \beta \cdot {x} \cdot {{\log}^{-\alpha}(x)}$ as $x$ tends to $\infty$ (for some constant $\beta > 0$). Conclude that \eqref{eq:bound} and the desired density zero conclusion hold (as $\mathcal{N}_n$ has positive density), thus ending the proof of Theorem \ref{thm:super_uncond}.

\subsection{Variants of Theorem \ref{thm:super_uncond}} \label{app:proof_2}

As before, we refer to \S\ref{ssec:basics_0} and \S\ref{ssec:basics_2}.3 for the definitions of the sets $\mathcal{P}(n,N)$, $\mathcal{P}(n,N, H)$, $\mathcal{N}_n$, $\mathcal{N}_n(P(T))$, $\mathcal{N}_n(x)$, and $\mathcal{N}_n(P(T), x)$.

\begin{proposition} \label{prop:super_uncond}
Let $n$ and $N$ be integers such that $n \geq 2$, $n$ is not a prime number, and $N \geq 5$. Let $P(T)$ be a separable polynomial in $\mathcal{P}(n,N)$ and let $n_1$ be the smallest prime divisor of $n$. Then there exists a positive constant $c$ such that 
\begin{equation} \label{eq:pnt}
|\mathcal{N}_n(x)| - |\mathcal{N}_n(P(T), x)| \geq c \cdot x^{1/n_1}, \quad x \rightarrow \infty.
\end{equation}
\end{proposition}

\begin{proof}
For $\alpha \in \mathcal{N}_2$, consider the $n$-free integer $d_\alpha = 2 \alpha^{n_1}$ (note that $n_1 \not \in \{n-1,n\}$ as $n$ is not a prime and $n_1 \, | \, n$). We show below that there are only finitely many squarefree integers
$\alpha$ such that the twisted superelliptic curve $C_{d_\alpha \cdot P(T)}: y^n = d_\alpha \cdot P(t)$ has a non-trivial $\Qq$-rational point, thus providing \eqref{eq:pnt}.

Set $n=n_1 n_2$. Given $\alpha \in \mathcal{N}_2$, let $[y_\alpha:t_\alpha:z_\alpha]$ be a non-trivial $\Qq$-rational point on $C_{d_\alpha \cdot P(T)}$. If $z_\alpha \not=0$, then $[y_\alpha : t_\alpha : z_\alpha]=[y_\alpha' :t_\alpha/z_\alpha :1]$, where $y_\alpha'$ is $y_\alpha$ divided by some power of $z_\alpha$. One may then assume $z_\alpha=0$ or $z_\alpha=1$. In each case, one sees that $[y_\alpha^{n_2}/\alpha : t_\alpha : z_\alpha]$ is a non-trivial $\Qq$-rational point on $C_{2 \cdot P(T)}:y^{n_1}= 2 \cdot P(t)$.

Now, given $\alpha \not= \beta \in \mathcal{N}_2$, suppose $[y_\alpha^{n_2}/\alpha : t_\alpha : z_\alpha] = [y_\beta^{n_2}/\beta : t_\beta : z_\beta]$. First, if $z_\alpha = z_\beta = 0$ (which implies that $n$ divides $N$), then $y_\alpha^{n_2}/\alpha = \lambda^{N/n_1} y_\beta^{n_2}/\beta$ for some $\lambda \in \Qq \setminus \{0\}$. Since $n_2$ divides $N/n_1$, this implies that $\alpha / \beta \not=1$ is a $n_2$th power in $\Qq$, which cannot happen. Now, if $z_\alpha = z_\beta=1$, then $y_\alpha^{n_2}/\alpha = y_\beta^{n_2}/\beta$ and one gets a contradiction as in the first case.

Hence, if $C_{d_\alpha \cdot P(T)}$ has a non-trivial $\Qq$-rational point for infinitely many $\alpha \in \mathcal{N}_2$, then $|C_{2 \cdot P(T)}(\Qq)| = \infty$. However, due to our assumptions that $P(T)$ is separable and $N \geq 5$, this superelliptic curve has genus at least 2 and Faltings' theorem then yields a contradiction.
\end{proof}

\begin{proposition} \label{prop:hyper_uncond}
Let $N$ be a positive integer such that $N \equiv 3 \pmod 4$. Then there exists a subset $S$ of $\mathcal{P}(2,N)$ which satisfies the following two conclusions.

\vspace{0.5mm}

\noindent
{\rm{(a)}} One has
\begin{equation} \label{eq:bound_again}
\frac{|{S} \cap \mathcal{P}(2,N,H)|}{|\mathcal{P}(2,N,H)|} = 1 - O \bigg(\frac{{\log}(H)}{\sqrt{H}} \bigg), \quad H \rightarrow \infty.
\end{equation}
In particular, the set ${S}$ has density 1.

\vspace{0.5mm}

\noindent
{\rm{(b)}} The complement $\mathcal{N}_2 \setminus \mathcal{N}_2(P(T))$ is infinite for every polynomial $P(T) \in S$.
\end{proposition}

\begin{proof}
See, e.g., the survey paper \cite{Sto14} for more details on the terminology we use below.

First, given $N \geq 1$ odd and a polynomial $P(T) \in \mathcal{P}(2,N)$, suppose there exists an infinite subset $\mathcal{S}$ of $\mathcal{N}_2$ such that the 2-Selmer group ${\rm{Sel}}_2(J(C_{d \cdot P(T)}))$ of the Jacobian $J(C_{d \cdot P(T)})$ of $C_{d \cdot P(T)}:y^2=d \cdot P(t)$ is trivial for each $d \in \mathcal{S}$. For such a $d$, denote the Mordell-Weil rank of $J(C_{d \cdot P(T)})$ by $r_d$ and the 2-torsion subgroup of $J(C_{d \cdot P(T)})(\Qq)$ by $J(C_{d \cdot P(T)})(\Qq)[2]$. Then
$$r_d \leq {\rm{dim}}_{\mathbb{F}_2} {\rm{Sel}}_2(J(C_{d \cdot P(T)})) - {\rm{dim}}_{\mathbb{F}_2} J(C_{d \cdot P(T)})(\Qq)[2] = - {\rm{dim}}_{\mathbb{F}_2} J(C_{d \cdot P(T)})(\Qq)[2]. $$
See \cite[\S3]{Sto14} for more details. Consequently, one has $r_d = {\rm{dim}}_{\mathbb{F}_2} J(C_{d \cdot P(T)})(\Qq)[2]=0.$
Moreover, up to dropping finitely many elements of $\mathcal{S}$, we may assume that
$$J(C_{d \cdot P(T)})(\Qq)[{\rm{tors}}] = J(C_{d \cdot P(T)})(\Qq)[2]$$
for each $d \in \mathcal{S}$, with $J(C_{d \cdot P(T)})(\Qq)[{\rm{tors}}]$ the set of torsion points of $J(C_{d \cdot P(T)})(\Qq)$. We refer to \cite[Theorem 2.1]{BCS17} for more details. Hence, every $\Qq$-rational point on $J(C_{d \cdot P(T)})$ is of order 1. In particular, for each $d \in \mathcal{S}$, the set $C_{d \cdot P(T)}(\Qq)$ is reduced to $[0:1:0]$.

Now, given $N$ such that $N \equiv 3 \pmod 4$, consider the subset $S$ of $\mathcal{P}(2,N)$ defined by the extra condition that the Galois group over $\Qq$ is $S_N$ or $A_N$. As in the proof of Theorem \ref{thm:super_uncond}, one shows that the set $S$ fulfills \eqref{eq:bound_again}. Moreover, by \cite[Theorem 3]{Yu16}, for every $P(T) \in S$, there exist infinitely many $d \in \mathcal{N}_2$ such that the 2-Selmer group of the Jacobian of $C_{d \cdot P(T)}$ is trivial. It then remains to apply the first part of the proof to conclude.
\end{proof}

\begin{remark}
(a) If $N=3$, one can take $S = \mathcal{P}(2,N)$. Indeed, for $P(T) \in \mathcal{P}(2,3)$, it is known that the Mordell-Weil rank of $C_{d \cdot P(T)}$ is 0 for infinitely many $d \in \mathcal{N}_2$, and that, for all but finitely many $d \in \mathcal{N}_2$, every torsion $\Qq$-rational point on $C_{d \cdot P(T)}$ is trivial. See, e.g., \cite{Dab08} and \cite[Proposition 1]{GM91} for more details and references. Moreover, for some $P(T) \in \mathcal{P}(2,3)$, the density of $\mathcal{N}_2 \setminus \mathcal{N}_2(P(T))$ is known to be positive (unconditionally). 
See \cite{Dab08} for references.

\vspace{1mm}

\noindent
(b) Given $r \geq 2$ even, the density of the subset $\mathcal{E}^{\infty}(r)$ of $\mathcal{E}(r)$ (see \S\ref{ssec:basics_3}), which consists of all regular $\Zz/2\Zz$-extensions of $\Qq(T)$ with exactly $r$ branch points and with $\infty$ as a branch point, is easily seen to be 0 by Proposition \ref{prop:card}. Consequently, elements of $\mathcal{E}^\infty(r)$ which are contained in a set $S$ as in Theorem \ref{thm:Z/2Z_even} are only a negligible part of $S$. However, if $r$ is divisible by 4, Proposition \ref{prop:hyper_uncond} shows that there is a density 1 subset $S$ of $\mathcal{E}^\infty(r)$ such that there exist infinitely many quadratic extensions of $\Qq$ which do not belong to the specialization set of a given extension of $\Qq(T)$ in $S$. The precise statement and the proof, which is very similar to that of Theorem \ref{thm:Z/2Z_even} under Theorem \ref{thm:super_uncond}, are left to the interested reader.
\end{remark}

\bibliography{Biblio2}

\begin{thebibliography}{DKLN18}

\bibitem[BCS17]{BCS17}
Abbey Bourdon, Pete Clark, and James Stankewicz.
\newblock {T}orsion points on {C}{M} elliptic curves over real number fields.
\newblock {\em Trans. Amer. Math. Soc.}, 369(12):8457--8496, 2017.

\bibitem[Bec91]{Bec91}
Sybilla Beckmann.
\newblock On extensions of number fields obtained by specializing branched
  coverings.
\newblock {\em J. Reine Angew. Math.}, 419:27--53, 1991.

\bibitem[Bha07]{Bha07}
Manjul Bhargava.
\newblock Mass formulae for extensions of local fields, and conjectures on the
  density of number field discriminants.
\newblock {\em Int. Math. Res. Not. IMRN}, no. 17, 2007.
\newblock Art. {I}{D} rnm052, 20 pp.

\bibitem[BJK09]{BJK09}
Dongho Byeon, Daeyeol Jeon, and Chang~Heon Kim.
\newblock Rank-one quadratic twists of an infinite family of elliptic curves.
\newblock {\em J. Reine Angew. Math.}, 633:67--76, 2009.

\bibitem[BW08]{BW08}
Manjul Bhargava and Melanie Wood.
\newblock The density of discriminants of ${S}_3$-sextic number fields.
\newblock {\em Proc. Amer. Math. Soc.}, 136(5):1581--1587, 2008.

\bibitem[Bye04]{Bye04}
Dongho Byeon.
\newblock Ranks of quadratic twists of an elliptic curve.
\newblock {\em Acta Arith.}, 114(4):391--396, 2004.

\bibitem[CHM97]{CHM97}
Lucia Caporaso, Joe Harris, and Barry~C. Mazur.
\newblock Uniformity of rational points.
\newblock {\em J. Amer. Math. Soc.}, 10(1):1--35, 1997.

\bibitem[Coh81]{Coh81b}
Stephen~D. Cohen.
\newblock The distribution of {G}alois groups and {H}ilbert's irreducibility
  theorem.
\newblock {\em Proc. London Math. Soc. (3)}, 43(2):227--250, 1981.

\bibitem[CW18]{CW18}
Pete Clark and Lori Watson.
\newblock {A}{B}{C} and the {H}asse principle for quadratic twists of
  hyperelliptic curves.
\newblock {\em C. R. Math. Acad. Sci. Paris}, 356(9):911--915, 2018.

\bibitem[Dab08]{Dab08}
Andrzej Dabrowski.
\newblock On the proportion of rank 0 twists of elliptic curves.
\newblock {\em C. R. Math. Acad. Sci. Paris}, 346(9-10):483--486, 2008.

\bibitem[D{\`{e}}b99]{Deb99a}
Pierre D{\`{e}}bes.
\newblock Galois covers with prescribed fibers: the {B}eckmann-{B}lack problem.
\newblock {\em Ann. Scuola Norm. Sup. Pisa Cl. Sci. (4)}, 28(2):273--286, 1999.

\bibitem[D{\`e}b09]{Deb09}
Pierre D{\`e}bes.
\newblock {\em Arithm\'etique des rev\^etements de la droite}.
\newblock {L}ecture notes, 2009.
\newblock {A}t http://math.univ-lille1.fr/\~{}pde/ens.html.

\bibitem[D{\`e}b17]{Deb17}
Pierre D{\`e}bes.
\newblock On the {M}alle conjecture and the self-twisted cover.
\newblock {\em Israel J. Math.}, 218(1):101--131, 2017.

\bibitem[D{\`e}b18]{Deb18}
Pierre D{\`e}bes.
\newblock Groups with no parametric {G}alois realizations.
\newblock {\em Ann. Sci. \'{E}c. Norm. Sup\'{e}r. (4)}, 51(1):143--179, 2018.

\bibitem[DF90]{DF90}
Pierre D{\`e}bes and Michael~D. Fried.
\newblock Rigidity and real residue class fields.
\newblock {\em Acta Arith.}, 56(4):291--323, 1990.

\bibitem[DG12]{DG12}
Pierre D{\`e}bes and Nour Ghazi.
\newblock Galois covers and the {H}ilbert-{G}runwald property.
\newblock {\em Ann. Inst. Fourier (Grenoble)}, 62(3):989--1013, 2012.

\bibitem[DKLN18]{DKLN18}
Pierre D\`{e}bes, Joachim K\"{o}nig, Fran\c{c}ois Legrand, and Danny Neftin.
\newblock Rational pullbacks of {G}alois covers.
\newblock {\em {M}anuscript}, 2018.
\newblock arXiv 1807.01937.

\bibitem[DLAN17]{DLN17}
Cyril Demarche, Giancarlo Lucchini~Arteche, and Danny Neftin.
\newblock The {G}runwald problem and approximation properties for homogeneous
  spaces.
\newblock {\em Ann. Inst. Fourier (Grenoble)}, 67(3):1009--1033, 2017.

\bibitem[FJ08]{FJ08}
Michael~D. Fried and Moshe Jarden.
\newblock {\em Field arithmetic}.
\newblock Ergebnisse der Mathematik und ihrer Grenzgebiete. 3. Folge. A
  {S}eries of Modern Surveys in Mathematics [Results in Mathematics and Related
  Areas. 3rd Series. A Series of Modern Surveys in Mathematics], 11.
  Springer-Verlag, Berlin, third edition, 2008.
\newblock Revised by Jarden. xxiv + 792 pp.

\bibitem[Fri77]{Fri77}
Michael~D. Fried.
\newblock Fields of definition of function fields and {H}urwitz families-groups
  as {G}alois groups.
\newblock {\em Comm. Algebra}, 5(1):17--82, 1977.

\bibitem[GKZ94]{GKZ94}
Izrail'~Moiseevich Gel'fand, Mikhail~M. Kapranov, and Andrey~V. Zelevinsky.
\newblock {\em Discriminants, resultants, and multidimensional determinants}.
\newblock Mathematics: {T}heory \& {A}pplications. Birkh\"{a}user Boston, Inc.,
  Boston, MA, 1994.
\newblock x+523 pp.

\bibitem[GM91]{GM91}
Fernando~Quadros Gouv\^ea and Barry~C. Mazur.
\newblock The square-free sieve and the rank of elliptic curves.
\newblock {\em J. Amer. Math. Soc.}, 4(1):1--23, 1991.

\bibitem[Gra98]{Gra98}
Andrew Granville.
\newblock {\it{{A}{B}{C}}} allows us to count squarefrees.
\newblock {\em Internat. Math. Res. Notices}, 19:991--1009, 1998.

\bibitem[Gra07]{Gra07}
Andrew Granville.
\newblock Rational and integral points on quadratic twists of a given
  hyperelliptic curve.
\newblock {\em Int. Math. Res. Not. IMRN}, no. 8, 2007.
\newblock Art. {I}{D} 027, 24 pp.

\bibitem[JLY02]{JLY02}
Christian~U. Jensen, Arne Ledet, and Noriko Yui.
\newblock {\em Generic polynomials. Constructive Aspects of the Inverse Galois
  Problem}.
\newblock {M}athematical {S}ciences {R}esearch {I}nstitute {P}ublications, 45.
  Cambridge University Press, 2002.
\newblock x+258 pp.

\bibitem[KL18]{KL18}
Joachim K\"onig and Fran\c{c}ois Legrand.
\newblock Non-parametric sets of regular realizations over number fields.
\newblock {\em J. Algebra}, 497:302--336, 2018.

\bibitem[KLN19]{KLN19}
Joachim K\"onig, Fran\c{c}ois Legrand, and Danny Neftin.
\newblock On the local behavior of specializations of function field
  extensions.
\newblock {\em Int. Math. Res. Not. IMRN}, 2019(9):2951--2980, 2019.

\bibitem[Kl{\"u}06]{Klu06}
J{\"u}rgen Kl{\"u}ners.
\newblock Asymptotics of number fields and the {C}ohen-{L}enstra heuristics.
\newblock {\em J. Th\'eor. Nombres Bordeaux}, 18(3):607--615, 2006.

\bibitem[KM01]{KM01}
J\"urgen Kl\"uners and Gunter Malle.
\newblock A database for field extensions of the rationals.
\newblock {\em LMS J. Comput. Math.}, 4:182--196, 2001.

\bibitem[KM04]{KM04}
J\"urgen Kl\"{u}ners and Gunter Malle.
\newblock Counting nilpotent {G}alois extensions.
\newblock {\em J. Reine Angew. Math.}, 572:1--26, 2004.

\bibitem[K{\"{o}}n17]{Koe17a}
Joachim K{\"{o}}nig.
\newblock Non-parametricity of rational translates of regular {G}alois
  extensions.
\newblock {\em Acta. Arith.}, 179(3):267--275, 2017.

\bibitem[Leg15]{Leg15}
Fran\c{c}ois Legrand.
\newblock Parametric {G}alois extensions.
\newblock {\em J. Algebra}, 422:187--222, 2015.

\bibitem[Leg16]{Leg16}
Fran\c{c}ois Legrand.
\newblock Specialization results and ramification conditions.
\newblock {\em Israel J. Math.}, 214(2):621--650, 2016.

\bibitem[Leg18]{Leg18b}
Fran\c{c}ois Legrand.
\newblock Twists of superelliptic curves without rational points.
\newblock {\em Int. Math. Res. Not. IMRN}, 2018(4):1153--1176, 2018.

\bibitem[Mal02]{Mal02}
Gunter Malle.
\newblock On the distribution of {G}alois groups.
\newblock {\em J. Number Theory}, 92(2):315--329, 2002.

\bibitem[MM18]{MM18}
Gunter Malle and B.~Heinrich Matzat.
\newblock {\em {I}nverse {G}alois theory}.
\newblock {S}pringer {M}onographs in {M}athematics. Springer, Berlin, 2018.
\newblock Second edition. xvii+532 pp.

\bibitem[Ser76]{Ser76}
Jean-Pierre Serre.
\newblock Divisibilit{\'{e}} de certaines fonctions arithm{\'{e}}tiques.
\newblock {\em Enseignement Math. (2)}, 22(3-4):227--260, 1976.

\bibitem[Ser92]{Ser92}
Jean-Pierre Serre.
\newblock {\em Topics in Galois Theory}, volume~1 of {\em Research Notes in
  Mathematics}.
\newblock Jones and Bartlett Publishers, Boston, MA, 1992.
\newblock Lecture notes prepared by {H}enri {D}armon [{H}enri {D}armon]. With a
  foreword by {D}armon and the author. xvi+117 pp.

\bibitem[Sto14]{Sto14}
Michael Stoll.
\newblock Rational points on hyperelliptic curves: {R}ecent developments.
\newblock In {\em Computeralgebra-Rundbrief 54}. Fachgruppe Computeralgebra,
  2014.

\bibitem[Vat98]{Vat98}
Vinayak Vatsal.
\newblock Rank-one twists of a certain elliptic curve.
\newblock {\em Math. Ann.}, 311(4):791--794, 1998.

\bibitem[V{\"o}l96]{Vol96}
Helmut V{\"o}lklein.
\newblock {\em Groups as {G}alois groups. An introduction}, volume~53 of {\em
  Cambridge Studies in Advanced Mathematics}.
\newblock Cambridge University Press, Cambridge, 1996.
\newblock xviii+248 pp.

\bibitem[Yu16]{Yu16}
Myungjun Yu.
\newblock Selmer ranks of twists of hyperelliptic curves and superelliptic
  curves.
\newblock {\em J. Number Theory}, 160:148--185, 2016.

\end{thebibliography}
\bibliographystyle{alpha}

\end{document}